\documentclass{amsart}
\usepackage{graphicx} % Required for inserting images
\usepackage{amsmath,amssymb,amsxtra,anysize,adjustbox,xcolor,combelow,hyperref} % Required for inserting images
\usepackage{mathrsfs}  
\usepackage{tikz-cd,comment}
\usetikzlibrary{calc, babel}%
\usetikzlibrary{shapes}%
\usetikzlibrary{patterns}%
\usetikzlibrary{positioning}%
\usetikzlibrary{arrows.meta}
\usetikzlibrary{knots}
\usetikzlibrary{decorations.markings}
\usetikzlibrary{hobby}
\tikzset{curve/.style={settings={#1},to path={(\tikztostart)
    .. controls ($(\tikztostart)!\pv{pos}!(\tikztotarget)!\pv{height}!270:(\tikztotarget)$)
    and ($(\tikztostart)!1-\pv{pos}!(\tikztotarget)!\pv{height}!270:(\tikztotarget)$)
    .. (\tikztotarget)\tikztonodes}},
    settings/.code={\tikzset{quiver/.cd,#1}
        \def\pv##1{\pgfkeysvalueof{/tikz/quiver/##1}}},
    quiver/.cd,pos/.initial=0.35,height/.initial=0}

\usepackage{braket}

\newcommand{\Bqt}{\mathbb{B}_{q,t}}

\newcommand{\C}{\mathbb{C}}

\newcommand{\Z}{\mathbb{Z}}

\DeclareMathOperator{\Rep}{Rep}

\newcommand{\nak}{\mathcal{M}}

\DeclareMathOperator{\Hom}{Hom}
\DeclareMathOperator{\End}{End}
\newcommand{\uk}{\underline{k}}
\newcommand{\ur}{\underline{r}}

\newcommand{\GL}{\mathrm{GL}}
\newcommand{\gl}{\mathfrak{gl}}
\newcommand{\dimvec}{\underline{d}}
\newcommand{\framvec}{\underline{f}}
\newcommand{\dimv}[1]{\dimvec^{(#1)}}
\newcommand{\double}[1]{\mathrm{D}\!\left(#1\right)}
\newcommand{\framed}[1]{#1^{\heartsuit}}
\newcommand{\rep}[2]{\widehat{#1}_{#2}}
\newcommand{\framedrep}[2]{\widehat{#1}_{#2, \heartsuit}}
\newcommand{\prequot}[2]{\mathcal{P}_{#1}\!\left(#2\right)}
\newcommand{\pnak}[2]{\mathcal{M}_{#1}\!\left(#2\right)}
\newcommand{\parnak}[3]{\mathcal{M}^{#1}_{#2}\!\left(#3\right)}
\newcommand{\affnak}[2]{^{0}\!\mathcal{M}_{#1}(#2)}
\newcommand{\can}{\delta}
\newcommand{\comp}{c}
\newcommand{\Hilb}{\mathrm{Hilb}}

\newcommand{\Gr}{\mathrm{Gr}}
\newcommand{\Fl}{\mathrm{Fl}}

\newcommand{\diff}[1]{\mathrm{diff}(#1)}
\newcommand{\dimvector}{\underline{\mathrm{dim}}}

\def\CK{\mathcal{K}}

\newcommand{\BN}{\mathrm{BN}}

\DeclareMathOperator{\spann}{span}

\newcommand{\Imm}{\mathrm{Im}}
\newcommand{\Tr}{\mathrm{Tr}}

\newcommand{\CL}{\mathcal{L}}
\newcommand{\CD}{\mathcal{D}}

\newcommand{\Lie}{\mathrm{Lie}}
\newcommand{\target}{\mathtt{in}}
\newcommand{\source}{\mathtt{out}}

\newcommand{\st}{\mathrm{st}}
\newcommand{\Cartan}{\mathbf{C}}

\DeclareMathOperator{\rk}{rk}
\newcommand{\reg}{\mathrm{reg}}

\usepackage{enumitem}

\usepackage{mathtools}

\numberwithin{equation}{section}

\newtheorem{theorem}{Theorem}[section]
\newtheorem{proposition}[theorem]{Proposition}
\newtheorem{corollary}[theorem]{Corollary}
\newtheorem{lemma}[theorem]{Lemma}

\theoremstyle{definition}
\newtheorem{remark}[theorem]{Remark}

\newtheorem{definition}[theorem]{Definition}
\newtheorem{example}[theorem]{Example}

\title{Carlsson-Mellit algebra for arbitrary quivers}

\author{Nicolle Gonz\'alez}
\address{Department of Mathematics, The University of British Columbia,
Vancouver, BC Canada}
\email{nicolle@math.ubc.ca}

\author{Eugene Gorsky}
\address{Department of Mathematics, University of California Davis\\ One Shields Avenue, Davis CA 95616 USA}
\email{egorskiy@ucdavis.edu}

\author{Jos\'e Simental}
\address{Instituto de Matemáticas, Universidad Nacional Autónoma de México \\ Ciudad Universitaria, CDMX, México }
\email{simental@im.unam.mx}

\title{Smooth correspondences between quiver varieties}


\date{}

\begin{document}

\begin{abstract}
We introduce a new class of smooth correspondences between Nakajima quiver varieties called split parabolic quiver varieties, and study their properties. We use these correspondences to construct an explicit resolution of singularities of quiver Brill--Noether loci and prove that the latter are  irreducible and Cohen-Macaulay of expected dimension (if non-empty). This generalizes 
the results of Nakajima--Yoshioka and  Bayer--Chen--Jiang for Hilbert schemes of points on surfaces.
\end{abstract}

\maketitle

\section{Introduction}

\subsection{Quiver varieties}
 
Nakajima quiver varieties constitute a cornerstone of geometric representation theory. Initially introduced by Nakajima \cite{nakajima-instantons} in order to geometrically construct integrable irreducible representations of Kac-Moody algebras, their different geometric invariants (such as (equivariant) homology or $K$-theory) afford representations of quantum algebras such as Yangians, quantum affine algebras, or $K$-theoretic Hall algebras of quivers, see e.g. \cite{nakajima-quiver, nakajima-fixed, varagnolo-yangians, NSS}. From a purely geometric point of view, quiver varieties make an important source of examples and simultaneously generalize cotangent bundles to (partial) flag varieties, type A Slodowy varieties, affine Grassmannian slices, ADHM spaces, and the Hilbert scheme of points in the plane, see \cite{ginzburg, maffei, mirkovic-vibornov, nakajima-book}.

A quiver variety $\nak_Q(\dimvec; \framvec)$ depends on the data of a quiver $Q = (Q_0, Q_1, \source, \target)$, where $Q_0$ is the set of vertices, $Q_1$ the set of arrows, and $\source, \target: Q_1 \to Q_0$ the functions that to each vertex associate its source and target vertices, respectively. We take two dimension vectors $\dimvec, \framvec \in \Z_{\geq 0}^{Q_0}$, that we think of as column vectors. The vector $\framvec$ will be referred to as the \emph{framing vector}, while $\dimvec$ will simply be referred to as the \emph{dimension vector}. In a nutshell, the variety $\nak_Q(\dimvec; \framvec)$ is defined to be the GIT Marsden-Weinstein (aka Hamiltonian) reduction of the symplectic vector space of representations of the framed double quiver $\double{\framed{Q}}$ with a natural action of the group $\GL_{\dimvec} := \prod_{v \in Q_0} \GL(d_v)$, see Section \ref{sec:quiver-varieties} for a precise definition and details on the construction. In particular, the GIT construction involves the choice of a \emph{stability condition}, and in this paper we only work with a specific choice of stability condition. 

An important role in the geometric representation theory of quiver varieties \cite{nakajima-km, varagnolo-yangians, NSS} is played by the smooth \emph{Nakajima correspondences} $N_Q(\dimvec; \dimvec - \can_v; \framvec)$ between the two quiver varieties $\nak_Q(\dimvec; \framvec)$ and $\nak_Q(\dimvec - \can_v; \framvec)$, where $\can_v \in \Z_{\geq 0}^{Q_0}$ is the vector that has $1$ in the $v$-th position and zeroes elsewhere. In this paper, we generalize the construction of Nakajima correspondences and define a family of smooth correspondences between the quiver varieties $\nak_Q(\dimvec; \framvec)$ and $\nak_Q(\dimvec - \uk; \framvec)$ for arbitrary $\uk \in \Z^{Q_0}_{\geq 0}$. More precisely, given a tuple of dimension vectors $\dimv{0}, \dots, \dimv{\ell}$ and a framing vector $\framvec$, we define a smooth variety
\[
\pnak{Q}{\dimv{0}, \dots, \dimv{\ell}; \framvec} \hookrightarrow \nak_{Q}(\dimv{0}; \framvec) \times \cdots \times \nak_Q(\dimv{\ell}; \framvec) 
\]
which we call the \emph{split parabolic quiver variety}. When $\dimv{0} = \dimvec$, $\dimv{1} = \dimvec - \can_v$, and the quiver $Q$ has no loops at the vertex $v$, the variety $\pnak{Q}{\dimv{0}, \dimv{1}; \framvec}$ coincides with the Nakajima correspondence $N_Q(\dimvec; \dimvec - \can_v; \framvec)$, see Proposition \ref{prop: Nakajima correspondence intro} for a more general statement in case of loops at $v$.

Our goal in this paper is to study the geometry of the split parabolic quiver varieties $\pnak{Q}{\dimv{0}, \dots, \dimv{\ell}; \framvec}$. In particular, we will show that they are smooth, and show that they provide natural resolutions of singularities of varieties motivated by Brill-Noether theory \cite{Bayer, nakajima-km, yamakawa}. In the  follow-up paper \cite{BqtQ} we  study the geometric representation theory of split parabolic quiver varieties. In particular, their equivariant $K$-theory affords representations of generalizations of the Carlsson-Mellit algebra $\Bqt$ \cite{CM, CGM} to arbitrary quivers. The definition of such a generalization is the main objective of \cite{BqtQ}.

\subsection{Brill-Noether theory for quiver varieties}
Our construction is in part motivated by the Brill-Noether theory which we review first. If $Q$ is a quiver, we denote by $\framed{Q}$ its framed version, and by $\double{\framed{Q}}$ the double of $\framed{Q}$, see Section \ref{sec:quiver-varieties} for details. For $v$ a vertex of $Q$ (i.e., not a framing vertex) we denote by $\C_v$ the one-dimensional representation of $\double{\framed{Q}}$ supported at $v$ where all paths of positive length act by zero.

\begin{definition}\label{def:BN-intro}
Let $\uk \in \Z_{\geq 0}^{Q_0}$ be a dimension vector, and assume $k_v \leq d_v$ for every $v \in Q_0$. The \emph{quiver Brill-Noether locus} $\BN_Q^{\uk}(\dimvec;\framvec)$ is the subvariety of $\nak_Q(\dimvec;\framvec)$ defined by
$$
\BN_Q^{\uk}(\dimvec;\framvec)=\left\{D\in \nak_Q(\dimvec;\framvec)\mid \dim \Hom_{\double{\framed{Q}}}(\C_v,D)\ge k_v\ \mathrm{for\ all}\ v\right\}.
$$
\end{definition}

It is easy to see that $\BN_Q^{\uk}(\dimvec;\framvec)$ is closed in $\nak_Q(\dimvec;\framvec)$. When $\uk=k\delta_v$ is supported at one vertex $v$, the quiver Brill Noether loci were considered in \cite[Section 4]{nakajima-km}. More general Brill-Noether loci can be obtained as intersections of these:
\begin{equation}
\label{eq: BN as intersection}
\BN_Q^{\uk}(\dimvec;\framvec)=\bigcap_{v}\BN_Q^{k_v\delta_v}(\dimvec;\framvec).
\end{equation}

\begin{remark}
\label{rem: Nakajima BN}
    More precisely, \cite[(4.3)]{nakajima-km} defines varieties by bounding  the dimension of $\Hom_{\double{\framed{Q}}}(D, \C_v)$ \emph{from above}.  The change from $\Hom_{\double{\framed{Q}}}(\C_v, D)$ to $\Hom_{\double{\framed{Q}}}(D, \C_v)$ is explained by the choice of stability condition: our choice of stability condition differs from that in \cite{nakajima-km} by a sign. Note also that while the variety $\BN^{\uk}_{Q}(\dimvec; \framvec)$ is closed, those varieties considered in \cite{nakajima-km} are open in $\nak_{Q}(\dimvec; \framvec)$.
\end{remark}

Let $Q_1(v,w)$ denote the set of edges from $v$ to $w$ in the quiver $Q$. In Lemma \ref{lem: k bound} we show that if $D$ is a stable representation then  
\begin{equation}\label{eq:def-k0-intro}
\dim \Hom(\C_v,D)\ge k^0_v,\quad \text{where} \quad k^0_v:=2d_v-\sum _{w}\left(|Q_1(v,w)|+|Q_1(w,v)|\right)d_w-f_v,
\end{equation}
so without loss of generality one can assume $k_v\ge k_v^0$. We can now state our first main result.

\begin{theorem}
\label{thm: BN intro}
Assume that $k_v\ge k_v^0$ for all $v$ and $\BN_Q^{\uk}(\dimvec;\framvec)$ is not empty. Then it is an irreducible  Cohen-Macaulay variety of (expected) dimension 
$$
\dim \BN_Q^{\uk}(\dimvec;\framvec)=\dim \nak_Q(\dimvec;\framvec)-\sum_{v}k_v(k_v-k^0_v).
$$
\end{theorem}
\begin{remark}
For $\uk=k\delta_v$ supported at one vertex this is proven in \cite[Proposition 4.5]{nakajima-km}, modulo comparison as in Remark \ref{rem: Nakajima BN}. 
Theorem \ref{thm: BN intro} shows that the intersection  \eqref{eq: BN as intersection} also has the expected dimension.
\end{remark}

For the Hilbert scheme of points on the plane, the Brill-Noether locus admits a more explicit description. If $Q$ is the Jordan quiver then
$$
\BN^k_Q(d;1)=\{I\in \Hilb^n(\C^2)\mid \dim\Hom_{\C[x,y]}(\C,\C[x,y]/I)\ge k\}=
$$
$$
\{I\in \Hilb^n(\C^2)\mid \dim \C\otimes_{\C[x,y]}I\ge k+1\}.
$$
Here $\C=\C[x,y]/(x,y)$ is the skyscraper sheaf at the origin. In this case, Theorem \ref{thm: BN intro} (and its generalization to $\Hilb^d(S)$ for an arbitrary smooth surface $S$) was proved in \cite{Bayer, NakajimaYoshioka2}. 

\subsection{Split parabolic quiver varieties}

To prove Theorem \ref{thm: BN intro}, we introduce a collection of varieties $\pnak{Q}{\dimvec^{(0)},\ldots,\dimvec^{(\ell)};\framvec}$ which we call \emph{split parabolic quiver varieties}. The precise definition is somewhat technical (see Definitions \ref{def: parabolic-variety-1} and \ref{def: parabolic-variety-2}) but the key ingredients are the following:
\begin{itemize}
\item We have $d^{(0)}_v\ge d^{(1)}_v\ge \ldots\ge d^{(\ell)}_v\ge 0$ for all $v$.
\item For each $m=0,\ldots,\ell$, a representation $D^{(m)}$ of $\double{\framed{Q}}$ with vector spaces $D^{(m)}_v$ at each vertex and linear maps $A_e^{(m)},B_e^{(m)},i_v^{(m)},j_v^{(m)}$ satisfying moment map equations and stability conditions.  
\item The dimension vector of $D^{(m)}$ is given by $\dimvec^{(m)}$,  the framing spaces $F_v$ have dimension $f_v$ and are independent of $m$.
\item A collection of surjective maps $\pi^{(m)}_v:D^{(m)}_v\twoheadrightarrow D^{(m+1)}_v$ for all $m=0,\ldots,\ell-1$. We require that $\pi^{(m)}_v$ commute with $A_e^{(m)},B_e^{(m)},i_v^{(m)},j_v^{(m)}$ in appropriate sense.
\item Last but not least, we consider the composition $\pi_v=\pi^{(\ell-1)}_v\circ \cdots \circ \pi^{(0)}_v:D^{(0)}_v\twoheadrightarrow D^{(\ell)}_v$. We require that the operators $A_e^{(0)},B_e^{(0)},j_v^{(0)}$ vanish on the kernel of $\pi_v$ for all $v$.
\end{itemize}

Our main structural result about these varieties reads as follows.

\begin{theorem}
\label{thm: parabolic intro}
If non-empty, $\pnak{Q}{\dimvec^{(0)},\ldots,\dimvec^{(\ell)};\framvec}$ is a smooth connected quasi-projective variety.
\end{theorem}

\begin{remark}
    In the case when $Q$ is the Jordan quiver, the split parabolic quiver variety has already appeared in work of Nakajima--Yoshioka \cite{NakajimaYoshioka1, NakajimaYoshioka2} and of Bayer--Chen--Jiang \cite{Bayer}, where it is denoted by $\widehat{M}^0(c)$ and $\Hilb^{\dagger}$, respectively. In particular, Theorem \ref{thm: parabolic intro} is proved in \cite{NakajimaYoshioka2} in this case.
\end{remark}

\begin{remark}
It is known \cite{crawley2001geometry} when the Nakajima quiver variety $\nak_Q(\dimvec;\framvec)$ is not empty. In Corollary \ref{cor: nonempty} we give a somewhat involved combinatorial algorithm determining when $\pnak{Q}{\dimvec^{(0)},\ldots,\dimvec^{(\ell)};\framvec}$ is not empty.
\end{remark}

We give an explicit formula for the dimension of $\pnak{Q}{\dimvec^{(0)},\ldots,\dimvec^{(\ell)};\framvec}$, and an explicit description of the tangent space at every point.
The split parabolic quiver varieties generalize classical Nakajima correspondences \cite{nakajima-km} in the following sense.
\begin{proposition}
\label{prop: Nakajima correspondence intro}
Let $N_{Q}(\dimvec,\dimvec-\delta_v;\framvec)$ be the Nakajima correspondence between $\nak_{Q}(\dimvec;\framvec)$ and $\nak_Q(\dimvec-\delta_v;\framvec)$. Then 
$$
N_{Q}(\dimvec,\dimvec-\delta_v;\framvec)\simeq \pnak{Q}{\dimvec,\dimvec-\delta_v;\framvec}\times \C^{2|Q_1(v,v)|}.
$$
\end{proposition}

The next easy observation shows that for many purposes it is sufficient to consider these varieties for $\ell=1$.

\begin{proposition}
\label{prop: flag fibration intro}
The natural projection 
$$
\pnak{Q}{\dimvec^{(0)},\ldots,\dimvec^{(\ell)};\framvec}\to \pnak{Q}{\dimvec^{(0)},\dimvec^{(\ell)};\framvec}
$$
is a locally trivial fibration with fiber the product of partial flag varieties 
$$
\prod_{v}\Fl\left(d_v^{(\ell - 1)} - d_v^{(\ell)}, \dots, d_v^{(1)} - d_v^{(\ell)}; d_v^{(0)} - d_v^{(\ell)} \right).
$$
\end{proposition}

As a special case of Proposition \ref{prop: flag fibration intro} we get the following.

\begin{corollary}
\label{cor: full flag intro}
Assume that $\dimvec^{(i)}-\dimvec^{(i+1)}=\delta_{v_i}$ for all $i=0,\ldots,\ell-1$ and some $v_i$. Furthermore, let $\uk=\dimvec^{(0)}-\dimvec^{(\ell)}$. Then the space
$$
\parnak{\uk}{Q}{\dimvec;\framvec}:=\pnak{Q}{\dimvec^{(0)},\ldots,\dimvec^{(\ell)};\framvec}
$$
depends only on $\uk,\dimvec,\framvec$ but not on a specific sequence $\dimvec^{(i)}$. The natural projection
$$
\parnak{\uk}{Q}{\dimvec;\framvec}\to \pnak{Q}{\dimvec,\dimvec-\uk;\framvec}
$$
is a locally trivial fibration with fiber the product of complete flag varieties
$
\prod_{v}\Fl(k_v).
$
\end{corollary}

Finally, we prove that split parabolic quiver varieties yield resolutions of singularities for Brill-Noether loci.

\begin{theorem}
\label{thm: BN resolution intro}
As in Theorem \ref{thm: BN intro}, assume that $k_v\ge k^0_v$ for all $v$ and $\BN^{\uk}_{Q}(\dimvec;\framvec)$ is nonempty. Consider the natural projection
$$
p_{-}: \pnak{Q}{\dimvec,\dimvec-\uk;\framvec}\to \nak_Q(\dimvec;\framvec)
$$  
which sends a pair $(D^{(0)},D^{(1)})$ to $D^{(0)}$. Then  $p_{-}$ is a resolution of singularities of the Brill-Noether locus $\BN^{\uk}_{Q}(\dimvec;\framvec)\subset \nak_Q(\dimvec;\framvec)$, which is the image of $p_{-}$.
\end{theorem}

\begin{remark}
When $Q$ is the Jordan quiver, this is essentially contained in \cite[Proposition 3.23]{NakajimaYoshioka2}
\end{remark}

\subsection{Organization of the paper}

The paper is organized as follows. In Section \ref{sec: geometry} we give two definitions of split parabolic quiver varieties and prove that they are equivalent. Proposition \ref{prop: flag fibration intro} immediately follows from the definitions.
Furthermore, we study properties of split parabolic quiver varieties and in particular prove in Theorem \ref{thm: smooth} that these are always smooth, this thus proves most of Theorem \ref{thm: parabolic intro}.
We also define various natural maps $P_{+},P_{-}$ between  split parabolic quiver varieties, see Lemma \ref{lem: maps}. 

In Section  \ref{sec: BN} we define and study quiver Brill-Noether  loci. The key tool is a pair of projections $p_{+},p_{-}$, see \eqref{eq: def p and q}, which are special cases of $P_{+}$ and $P_{-}$ above. We identify the images of $p_{+}$ and $p_{-}$ and prove that these are stratified Grassmannian fibrations, estimate dimensions of strata and their preimages. Overall, this allows us to prove Theorems \ref{thm: BN intro} and \ref{thm: BN resolution intro}. In Theorem \ref{thm: irreducible} we prove that both quiver Brill-Noether   loci and split parabolic quiver varieties are irreducible, if not empty. As a byproduct of the proof, we show that any quiver Brill-Noether  locus or split parabolic quiver variety is birational to an iterated Grassmannian bundle over an open dense subset in some quiver variety. We finish this section with a sequence of examples of the varieties and maps introduced here. Finally, in Section \ref{sec: compositions-correspondences} we study the compositions of split parabolic quiver varieties when viewed as correspondences that, besides being of intrinsic interest, will form an essential part of the upcoming work \cite{BqtQ}.

\subsection{Future work}

In the follow-up paper \cite{BqtQ} we study the equivariant $K$-theory of the varieties $\parnak{\uk}{Q}{\dimvec;\framvec}$ from Corollary \ref{cor: full flag intro}. More precisely, given an arbitrary quiver $Q$, we define an algebra $\Bqt(Q)$ which generalizes the double Dyck path algebra $\Bqt$ of Carlsson and Mellit \cite{CM,CGM,GGS1,GGS2}. We conjecture that $\Bqt(Q)$ acts on the equivariant (with respect to the torus action from Section \ref{sec:torus-action}) $K$-theory of  $\parnak{\uk}{Q}{\dimvec;\framvec}$ and prove this conjecture in many cases, in particular when the torus action has isolated fixed points. We also investigate the relation between the spherical subalgebra of $\Bqt(Q)$ and the $K$-theoretic Hall algebra of $Q$ \cite{NSS}.

\section*{Acknowledgments}

The authors would like to thank Andrei Negu\cb{t} and Andrei Okounkov for their interest in this work and many valuable suggestions.  Special thanks to Hiraku Nakajima and Alexei Oblomkov for pointing us towards the references \cite{NakajimaYoshioka1,NakajimaYoshioka2}. We would also like to thank Arend Bayer, Hunter Dinkins and Ruslan Maksimau for helpful correspondence.

The work of E. G. was partially supported  by the NSF grant DMS-2302305. The work of J. S. was partially supported by UNAM's PAPIIT Grant IA101526 and SECIHTI Project C-F-2023-G-106.

This material is based upon work supported by the National Science Foundation under Grant No. DMS-1929284 while the authors were in residence at the Institute for Computational and Experimental Research in Mathematics in Providence, RI, during the Categorification and Computation in Algebraic Combinatorics semester program.

\section{Split parabolic quiver varieties}
\label{sec: geometry}

\subsection{Quiver varieties}\label{sec:quiver-varieties} We first carefully review the construction of Nakajima quiver varieties, following \cite{ginzburg}. None of the constructions or results in this subsection are new. 

Let $Q = (Q_0, Q_1, \source, \target)$ be a quiver, with vertex set $Q_0$ and arrow set $Q_1$, and $\source, \target: Q_1 \to Q_0$ the maps that to each arrow associate its source and target vertices, respectively. For $v, w \in Q_0$, we will denote by $Q_1(v,w)$ the set of arrows from $v$ to $w$ in $Q$. Let $\dimvec = (d_v)_{v \in Q_0} \in \Z_{\geq 0}^{Q_0}$, that will be referred to as a \emph{dimension vector}. Fixing vector spaces $(D_v)_{v \in Q_0}$
of dimension $\dim(D_v) = d_v$, we have the \emph{representation space}
\begin{equation}\label{eq:rep-space}
\Rep_Q(\dimvec) = \bigoplus_{e \in Q_1}\Hom\left(D_{\source(e)}, D_{\target(e)}\right).
\end{equation}

We will need several variants of this construction. 

\begin{enumerate}
\item First, we define the \emph{double quiver} $\double{Q}$, that has the same vertex set as $Q$ and all the arrows of $Q$, but we add an arrow $e^{\ast}: w \to v$ for every arrow $e:  v \to w$. Note that:
\[
\Rep_{\double{Q}}(\dimvec) = \bigoplus_{e \in Q_1} \Hom\left(D_{\source(e)}, D_{\target(e)}\right) \oplus \Hom\left(D_{\target(e)}, D_{\source(e)}\right).
\]
We will denote an element of $\Rep_{\double{Q}}$ by $(A, B) = (A_{e}, B_{e})_{e \in Q_1}$.

\item Second, we define the \emph{framed quiver} $\framed{Q}$, as follows. It contains all the vertices and arrows of $Q$, plus extra vertices $\overline{v}$ for $v \in Q$, together with a framing arrow $\overline{v} \to v$. Note that a dimension vector for $\framed{Q}$ consists of a pair of dimension vectors $\dimvec, \framvec$ for $Q$. We call $\dimvec$ the \emph{dimension vector} and $\framvec$ the \emph{framing vector}. Besides the vector spaces $D_v$, we also fix vector spaces $F_v$ of dimension $f_v$, $v \in Q_0$. The representation space is
\[
\Rep_{\framed{Q}}(\dimvec; \framvec) = \bigoplus_{e \in Q_1}\Hom\left(D_{\source(e)}, D_{\target(e)}\right) \oplus \bigoplus_{v \in Q_0} \Hom\left(F_v, D_v\right). 
\]
We denote an element of $\Rep_{\framed{Q}}(\dimvec; \framvec)$ by $(A, i) = (A_{e}, i_v)_{e \in Q_1, v \in Q_0}$. 

\item Last, we combine the two constructions above into the \emph{double framed quiver} $\double{\framed{Q}}$. Note that
\begin{multline}\label{eq:framed-reps-doubled}
\Rep_{\double{\framed{Q}}}( \dimvec; \framvec) = \bigoplus_{e \in Q_1} \left(\Hom
\left(D_{\source(e)}, D_{\target(e)}\right) \oplus \Hom\left(D_{\target(e)}, D_{\source(e)}\right)\right) \oplus \\
\bigoplus_{v \in Q_0}\left(\Hom\left(F_v, D_v\right) \oplus \Hom\left(D_v, F_v\right)\right). 
\end{multline}
We denote an element of $\Rep_{\double{\framed{Q}}}( \dimvec; \framvec)$ by $(A, B, i, j) = (A_{e}, B_{e}, i_v, j_v)_{e \in Q_1, v \in Q_0}$. 
\end{enumerate}

Using the trace pairing, we see that $\Rep_{\double{\framed{Q}}}(\dimvec; \framvec) = T^{\ast}\!\Rep_{\framed{Q}}(\dimvec; \framvec)$, which is a symplectic variety. The group $G_{\dimvec} := \prod_{v \in Q_0} \GL(D_v)$ acts on $\Rep_{\double{\framed{Q}}}(\dimvec; \framvec)$ by symplectomorphisms: if $g = (g_v)_{v \in Q_0} \in \GL_{\dimvec}$ then
\[
g.(A,B,i,j) = \left(g_{\target(e)}A_{e}g_{\source(e)}^{-1}, g_{\source(e)}B_{e}g_{\target(e)}^{-1}, g_vi_v, j_vg_v^{-1}\right).
\]
This action admits a moment map, which is a $G_{\dimvec}$-equivariant map $\mu: \Rep_{\double{\framed{Q}}}( \dimvec; \framvec) \to \gl_{\dimvec} = \bigoplus_{v \in Q_0}\gl(D_v)$ that can be described as follows: the $v$-th component of $\mu(A, B, i, j)$ is
\begin{equation}\label{eq:moment-map-vertex}
\sum_{\substack{e \in Q_1 \\ \target(e) = v}}A_{e}B_{e} - \sum_{\substack{e \in Q_1 \\ \source(e) = v}}B_{e}A_{e} + i_vj_v,
\end{equation}
and we simplify this by simply writing $\mu(A, B, i, j) = [A, B] + ij$. Thus, $\mu^{-1}(0)$ consists of all elements $(A, B, i, j)$ such that \eqref{eq:moment-map-vertex} is zero for every $v \in Q_0$. Note that the action of $G_{\dimvec}$ on $\Rep_{\double{\framed{Q}}}(\dimvec; \framvec)$ preserves $\mu^{-1}(0)$.

\begin{definition}
The \emph{affine quiver variety} is the categorical quotient
\[
\affnak{Q}{\dimvec; \framvec} := \mu^{-1}(0)/\!/G_{\dimvec} = \mathrm{Spec}\left(\C[\mu^{-1}(0)]^{G_{\dimvec}}\right).
\]
\end{definition}

By standard results in geometric invariant theory, $\affnak{Q}{\dimvec; \framvec}$ parametrizes closed $G_{\dimvec}$-orbits in $\mu^{-1}(0)$.  By definition, the regular part of the affine quiver variety $\affnak{Q}{\dimvec; \framvec}^{\reg}$ consists of those closed orbits in $\mu^{-1}(0)$ with trivial stabilizer. We warn the reader that $\affnak{Q}{\dimvec; \framvec}^{\reg}$ may be empty. 

To obtain a space with finer geometric properties, we use geometric invariant theory and restrict to a class of stable representations.

\begin{definition}\label{def: stability}
    We say that an element $(A, B, i, j) \in \Rep_{\double{\framed{Q}}}( \dimvec; \framvec)$ is stable if the only collection of subspaces $(D'_v \subseteq D_v)_{v \in Q_0}$ satisfying:
\begin{enumerate}
\item $A_{e}\left(D'_{\source(e)}\right) \subseteq D'_{\target(e)}$, $B_{e}\left(D'_{\target(e)}\right) \subseteq D'_{\source(e)}$ for every $e \in Q_1$, and
\item $i(F_v) \subseteq D'_v$ for every $v \in Q_0$
\end{enumerate}
is $D'_v = D_v$. We denote by $\Rep_{\double{\framed{Q}}}( \dimvec; \framvec)^{\st}$ the subset consisting of stable representations. It is a Zariski open set in $\Rep_{\double{\framed{Q}}}( \dimvec; \framvec)$. 
\end{definition}

We denote by $\mu^{-1}(0)^{\st}$ the set of stable elements in $\mu^{-1}(0)$. It is well-known (see e.g. \cite[Theorem 5.2.2]{ginzburg}) that the group $G_{\dimvec}$ acts freely on $\mu^{-1}(0)^{\st}$. 

\begin{definition} 
The \emph{Nakajima quiver variety} is
\begin{equation}\label{eq:quiver-variety}
\nak_Q(\dimvec; \framvec) := \mu^{-1}(0)^{\st}/G_{\dimvec}.
\end{equation}
\end{definition}

The variety $\nak_Q(\dimvec; \framvec)$ admits a projective map $\nak_Q(\dimvec; \framvec) \rightarrow \, \affnak{Q}{\dimvec; \framvec}$. By \cite[Theorem 5.2.2]{ginzburg}, when $\nak_Q(\dimvec; \framvec)$ is nonempty, it is a smooth and symplectic algebraic variety of dimension 
\begin{equation}
\label{eq: dim nakajima}
\dim \nak_Q(\dimvec; \framvec)=2\sum_{e \in Q_1}d_{\source{(e)}}d_{\target{(e)}}+2\sum_{v\in Q_0}(d_vf_v-d_v^2).
\end{equation}

It will be useful to rephrase \eqref{eq: dim nakajima} as follows. We define the \emph{Cartan matrix} of $Q$ by
\begin{equation}
\label{eq: def Cartan}
\Cartan=(\Cartan_{v,w})_{v,w \in Q_0},\quad 
\Cartan_{v,w}=2\delta_{vw}-|Q_1(v,w)|-|Q_1(w,v)|.
\end{equation}
Note that $\Cartan$ is manifestly symmetric. Then
\begin{equation}
\label{eq: dim nakajima rephrased}
\dim \nak_Q(\dimvec; \framvec)=
2\framvec^T\dimvec-\dimvec^T\Cartan\dimvec.
\end{equation}
where here and below, if $M$ is a matrix then $M^{T}$ is its transpose. 

\begin{remark}\label{rmk: stability}
    More generally, each character $\theta: \GL_{\dimvec} \to \C^{\times}$ defines a stability condition on the representation space $\Rep_{\double{\framed{Q}}}(\dimvec; \framvec)$ and we have the Nakajima quiver variety
    \[
    ^{\theta}\!\nak_{Q}(\dimvec; \framvec) := \mu^{-1}(0)^{\theta-\st}/G_{\dimvec}.
    \]
    The stability condition in Definition \ref{def: stability} corresponds to the character $\theta(g_{v}) = \prod_{v}\det(g_v)^{-1}$. Taking duals, all the results and constructions in this paper have a straightforward analogue for the opposite stability condition corresponding to the character $\theta^{-1}(g_v) = \prod_{v} \det(g_v)$, see e.g. \cite[Section 5.1.4]{HC-typeA}. It would be interesting to define and study Brill-Noether loci and split parabolic quiver varieties for more general stability conditions, see Remark \ref{rmk: stability-2} below. 
\end{remark}

Finally, we recall an important torus action on $\nak(\dimvec; \framvec)$. We consider the torus
\[
\widetilde{T} := \widetilde{T}_{\framvec} = (\C^{\times}) \times (\C^{\times})^{Q_1} \times \prod_{v \in Q_0}\widetilde{T}_v,
\]
where $\widetilde{T}_v \subseteq \GL(f_v)$ is the torus of diagonal matrices. We denote an element of $\widetilde{T}$ by $(q, (t_e)_{e \in Q_1}, (U_v)_{v \in Q_0})$. The torus $\widetilde{T}$ acts on $\Rep_{\double{\framed{Q}}}( \dimvec; \framvec)$ as follows:
\begin{equation}\label{eq:torus-action}
(q, (t_e)_{e \in Q_1}, (U_v)_{v \in Q_0})\cdot (A, B, i, j) = (t_eA_e, q^{-1}t_e^{-1}B_e, q^{-1}i_vU_v^{-1}, U_vj_v).
\end{equation}
We note that under this action the moment map $\mu$ is homogeneous of weight $q^{-1}$, so the $\widetilde{T}$-action preserves the locus $\mu^{-1}(0)$. It is easy to verify that the $\widetilde{T}$-action preserves the locus of stable representations, and that it commutes with the action of $G_{\dimvec}$. Thus, we obtain an action of $\widetilde{T}$ on $\nak_Q(\dimvec; \framvec)$. Note that this action factors through an action
\[
T = T_{\framvec} \curvearrowright \nak_Q(\dimvec; \framvec)
\]
where $T_{\framvec} = \C^{\times} \times (\C^{\times})^{Q_1} \times \prod_{v \in Q_0}T_v,$
and $T_v$ is the quotient of $\widetilde{T}_v$ by the subgroup of scalar matrices. 

\begin{remark}\label{rmk:representations}
Given a quiver $Q$, one can form its path algebra $\C Q$, that has a basis consisting of all paths in $Q$, including lazy paths of length $0$, and multiplication is bilinearly extended from concatenation of paths. Given a finite-dimensional representation $D$ of $\C Q$, we denote $D_v = \varepsilon_v D$, where $\varepsilon_v \in \C Q$ is the lazy path of length $0$ at $v$. The dimension vector of $D$ is
\[
\dimvector(D) := (\dim_\C D_v)_{v \in Q_0} \in \Z_{\geq 0}^{Q_0}.
\]
Note that $D = \bigoplus_{v \in Q_0} D_v$ and that $\Rep_Q(\dimvec)$ is the algebraic variety parametrizing all representations of $\C Q$ with dimension vector $\dimvec$. Sometimes, we will want to consider the representation $D$ corresponding to an element in $\Rep_Q( \dimvec)$ and we will do so without further comment.

We note that $\mu^{-1}(0) \subseteq \Rep_{\double{Q}}(\dimvec)$ 
is the algebraic variety parametrizing representations of the \emph{preprojective algebra} $\Pi_Q$ of $Q$, which is the quotient of the path algebra $\C \double{Q}$ by the two-sided ideal generated by elements of the form $\sum_{e \in Q_1, \target(e) = v}ee^{\ast} - \sum_{e \in Q_1, \source(e) = v} e^{\ast}e$ where $v \in Q_0$. By definition, the category of preprojective algebra representations is a full subcategory of the category of representations of $\C \double{Q}$ which is closed under subquotients. 

Note that the Nakajima quiver variety $\nak_Q(\dimvec; \framvec)$ does not parametrize representations of the preprojective algebra $\Pi_{\framed{Q}}$, because we are not imposing the preprojective relations at the framing vertices. Nevertheless, $\nak_Q(\dimvec; \framvec)$ parametrizes stable representations of a quotient $\Pi$ of the path algebra $\C\double{\framed{Q}}$, and the category of $\Pi$-representations is a full subcategory of the category of $\C\double{\framed{Q}}$-representations which is closed under subquotients. 
\end{remark}

\subsection{Split parabolic quiver varieties} 
Now let $\uk = (k_v)_{v \in Q_0} \in \Z_{\geq 0}^{Q_0}$ be a vector, and assume that $k_v \leq d_v$ for every $v \in Q_0$. Let $\widehat{d}_v := d_{v} - k_v \geq 0$. Let $\widehat{D}_v$ be a vector space of dimension $\widehat{d}_v$, and fix a surjective map
\[
\pi_v: D_v \twoheadrightarrow \widehat{D}_v.
\]
Let $K_v := \ker(\pi_v)$, so that $\dim(K_v) = k_v$, and let $K := \bigoplus_{v \in Q_0} K_v$. We will consider the vector space
\[
\Rep^{\uk}_{\double{\framed{Q}}}(\dimvec; \framvec) = \bigoplus_{e \in Q_1}\Hom\left(\widehat{D}_{\source(e)}, D_{\target(e)}\right) \oplus \Hom\left(\widehat{D}_{\target(e)}, D_{\source(e)}\right) \oplus \bigoplus_{v \in Q_0} \Hom\left(F_v, D_v\right) \oplus \Hom\left(\widehat{D}_v, F_v\right). 
\]
We will denote an element of $\Rep^{\uk}_{\double{\framed{Q}}}\left(\dimvec; \framvec\right)$ by $\left(\widehat{A}_e, \widehat{B}_e, i_v, \widehat{j}_v\right).$

\begin{lemma}\label{lem: geometric embedding}
The map
\begin{equation}\label{eq: geometric embedding}
\begin{array}{rcl}
\iota: \Rep_{\double{\framed{Q}}}^{\uk}\left(\dimvec; \framvec\right) & \to & \Rep_{\double{\framed{Q}}}\left(\dimvec; \framvec\right) \\
\left(\widehat{A}_e, \widehat{B}_e, i_v, \widehat{j}_v\right) & \mapsto &  \left(\widehat{A}_e\pi_{\source(e)}, \widehat{B}_e\pi_{\target(e)}, i_v, \widehat{j}_v\pi_v\right)
\end{array}
\end{equation}
is injective, and its image consists of the set of all $\double{\framed{Q}}$-modules $D \oplus F$ that admit $K \oplus 0$ as a subrepresentation, where the latter is the representation where all paths of positive length act by zero.
\end{lemma}

\begin{proof}
    The map $\iota$ is injective because all maps $\pi_{v}$ are surjective. If $(A, B, i, j)$ is in the image of $\iota$, then for every edge $e \in Q_1$,
    \[
    A_e\left(K_{\source(e)}\right) = \widehat{A}_e\pi_{\source(e)}\left(K_{\source(e)}\right) = 0, \quad B_e\left(K_{\target(e)}\right) = 0, \quad j_v(K_v) = 0,
    \]
    so that $K \oplus 0$ is indeed a subrepresentation of $D \oplus F$. Conversely, if $K \oplus 0$ is a subrepresentation of $D \oplus F$ then, for every $e \in Q_1$, $A_e$ factors through $D_{\source(e)}/K_{\source(e)} \cong \widehat{D}_{\source(e)}$, $B_e$ factors through $D_{\target(e)}/K_{\target(e)} \cong \widehat{D}_{\target(e)}$, and $j_v$ factors through $D_v/K_v \cong \widehat{D}_v$, so it falls in the image of $\iota$. 
\end{proof}

Using the surjection $\pi_v$, we have an embedding $\Hom\left(\widehat{D}_v, D_v\right) \hookrightarrow \mathfrak{gl}(D_v)$. Note that $\mu \circ \iota: \Rep_{\double{\framed{Q}}}(\dimvec; \framvec) \to \mathfrak{gl}_{\dimvec}$ has its image contained in $\bigoplus_{v} \Hom\left(\widehat{D}_v, D_v\right)$. Also note that $\iota\left(\widehat{A}, \widehat{B}, i, \widehat{j}\right)$ belongs to $\mu^{-1}(0)$ if and only if, for every vertex $v \in Q_0$,
\[
\sum_{\substack{e \in Q_1 \\ \target(e) = v}} \widehat{A}_e\pi_{\source(e)}\widehat{B}_e - \sum_{\substack{e \in Q_1 \\ \source(e) = v}} \widehat{B}_e\pi_{\target(e)}\widehat{A}_e + i_v\widehat{j}_v = 0 \in \Hom\left(\widehat{D}_v, D_v\right).
\]

\begin{definition}
We let
\[
\Rep_{\double{\framed{Q}}}^{\uk}\left(\dimvec; \framvec\right)^{\st}_{0} := \iota^{-1}\left(\mu^{-1}(0)^{\st}\right)
\]
be the set consisting of those elements $\left(\widehat{A}, \widehat{B}, i, \widehat{j}\right)$ such that $\iota\left(\widehat{A}, \widehat{B}, i, \widehat{j}\right)$ is stable and satisfies the moment map equations \eqref{eq:moment-map-vertex}. 
\end{definition}

For each $v \in Q_0$, we fix a basis of $K_v$, that we complete to a basis of $D_v$, so that we identify the group $\GL(D_v)$ with $\GL(d_v)$, the group of invertible $d_v\times d_v$-matrices. Inside $\GL(d_v)$, we consider the group $P_v$ consisting of $(k_v, d_v - k_v)$-block upper triangular matrices, with upper triangular $k_v \times k_v$-block. We let $P_{\uk} := \prod_{v \in Q_0} P_v \subseteq \GL_{\dimvec}$. 

Note that, since $P_v$ preserves $K_v$ for each $v \in Q_0$, it follows from Lemma \ref{lem: geometric embedding} that the image of $\iota: \Rep_{\double{\framed{Q}}}^{\uk}\left(\dimvec; \framvec\right) \to \Rep_{\double{\framed{Q}}}\left(\dimvec; \framvec\right)$ is stable under the action of $P_{\uk}$. We can also give a formula for the action of $P_{\uk}$ on $\Rep_{\double{\framed{Q}}}^{\uk}\left(\dimvec; \framvec\right)$. Note that we have a well-defined map $P_{v} \to \GL(\widehat{D}_v)$, $g _v\mapsto \widehat{g_v}$ satisfying $\pi_v g_v = \widehat{g_v}\pi_v$. Thus, defining the action $P_{\uk} \curvearrowright \Rep_{\double{\framed{Q}}}^{\uk}\left(\dimvec; \framvec\right)$ by
\[
(g_v) \cdot \left(\widehat{A}, \widehat{B}, i, \widehat{j}\right) = \left(g_{\target(e)}\widehat{A}_e\widehat{g_{\source(e)}}^{-1}, g_{\target(e)}\widehat{B}_e\widehat{g_{\source(e)}}^{-1}, g_vi_v, \widehat{j}_v\widehat{g_v}^{-1}\right) 
\]
makes the map $\iota$ $P_{\uk}$-equivariant. We are ready to define the varieties of interest.

\begin{definition}\label{def: parabolic-variety-1}
    We define the \emph{split parabolic quiver variety} 
    \[
    \parnak{\uk}{Q}{\dimvec; \framvec} := \Rep^{\uk}_{\double{\framed{Q}}}\left(\dimvec; \framvec\right)^{\st}_{0}/P_{\uk} \cong \left(\iota\left(\Rep^{\uk}_{\double{\framed{Q}}}(\dimvec; \framvec)\right) \cap \mu^{-1}(0)^{\st}\right)\!\Big/ P_{\uk}. 
    \]
\end{definition}

Note that the group $P_{\uk}$ is not reductive, so we will need some work to show that the variety $\parnak{\uk}{Q}{\dimvec; \framvec}$ is indeed well-defined. In fact, we will give an alternative (equivalent) definition that, although significantly more involved, shows that $\parnak{\uk}{Q}{\dimvec; \framvec}$ can also be constructed as a quotient of a quasi-projective algebraic variety by the action of a reductive group.

\subsection{Alternative definition} To give an alternative definition of $\parnak{\uk}{Q}{\dimvec; \framvec}$ we consider several auxiliary quivers, that are very similar to those considered in \cite[Section 2]{Maksimau}, except that we consider framed versions. Let $\ell \geq 0$. 

\begin{enumerate}
    \item We consider the $\ell$-th repetition quiver $\rep{Q}{\ell}$. Its set of vertices is $Q_0 \times \{0,1, \dots, \ell\}$. For each arrow $e: v \to w$ in $Q$, there are arrows $(e, 0): (v, 0) \to (w,0), \dots, (e, \ell): (v, \ell) \to w(\ell)$ in $\rep{Q}{\ell}$. Finally, for every $v \in Q_0$ and $i = 0, \dots, \ell -1$, there is an arrow $(p_v, i): (v, i) \to (v, i+1)$.
    \item We also consider the $\ell$-th framed repetition quiver $\framedrep{Q}{\ell}$. 
    Its arrows and vertices include those of the repetition quiver $\rep{Q}{\ell}$, together with an extra set of (framing) vertices $\overline{Q_0} = \{\overline{v} \mid v \in Q_0\}$, and arrows $\overline{v} \to (v,i)$ and $(v, i) \to \overline{v}$ for every $i = 0, \dots, \ell$. 
    \item Finally, we consider the  $\ell$-th framed repetition quiver of the double quiver $\framedrep{\double{Q}}{\ell}$. Note that there is no arrow opposite to the arrows $(p_v, i)$ in the repetition quiver.   
\end{enumerate}

\begin{remark}
We note that the quiver $\framedrep{Q}{\ell}$ does \emph{not} coincide with the quiver $(\rep{Q}{\ell})^{\heartsuit}$, or with the quiver $\rep{(\framed{Q})}{\ell}$. For example, if $Q$ is the quiver with a single vertex and no arrows we have:
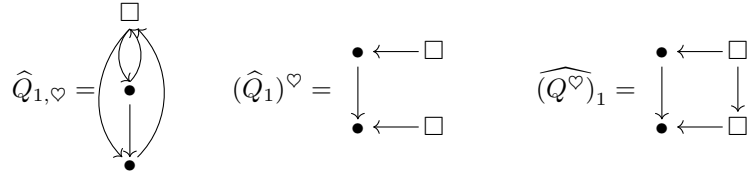
\begin{center}
    \begin{tikzcd}
\node at (0,0) {\framedrep{Q}{1} = };
\node at (1,1) {\square};
\node at (1,0) {\bullet};
\node at (1, -1) {\bullet};
\draw[->, out=225, in=135] (1,0.9) to (1,0.2);
\draw[->, out=45, in=315] (1,0.2) to (1,0.9);
\draw[->, out=225, in=135] (1,0.9) to (0.9,-0.8);
\draw[->, out=45, in=315] (1.1, -0.8) to (1.1,0.9);

\draw[->] (1, -0.1) to (1, -0.8);

\node at (3,0) {(\rep{Q}{1})^{\heartsuit} = };
\node at (4, 0.5) {\bullet}; \node at (5, 0.5) {\square};
\node at (4, -0.5) {\bullet}; \node at (5, -0.5) {\square};
\draw[->] (4, 0.4) to (4, -0.3);
\draw[->] (4.8, 0.6) to (4.2, 0.6);
\draw[->] (4.8, -0.4) to (4.2, -0.4);

\node at (7,0) {\widehat{(\framed{Q})}_{1} =};
\node at (8, 0.5) {\bullet}; \node at (9, 0.5) {\square};
\node at (8, -0.5) {\bullet}; \node at (9, -0.5) {\square};
\draw[->] (8, 0.4) to (8, -0.3);
\draw[->] (9, 0.4) to (9, -0.2);
\draw[->] (8.8, 0.6) to (8.2, 0.6);
\draw[->] (8.8, -0.4) to (8.2, -0.4);
    \end{tikzcd}
\end{center}
\end{remark}

\begin{example}
Consider the quiver $Q = \bullet \longrightarrow \bullet \longleftarrow \bullet$. The quiver $\rep{Q}{2}$ is
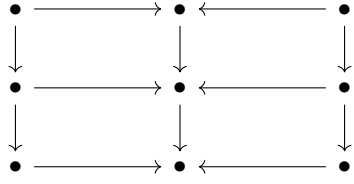
\[
\begin{tikzcd}
\bullet \arrow{rr} \arrow{d} & & \bullet \arrow{d} & & \bullet \arrow{ll} \arrow{d}
\\
\bullet \arrow{rr} \arrow{d} & & \bullet \arrow{d} & & \bullet \arrow{ll} \arrow{d}
\\
\bullet \arrow{rr}  & & \bullet  & & \bullet \arrow{ll} 
\end{tikzcd}
\]
while the quiver $\framedrep{Q}{2}$ is 
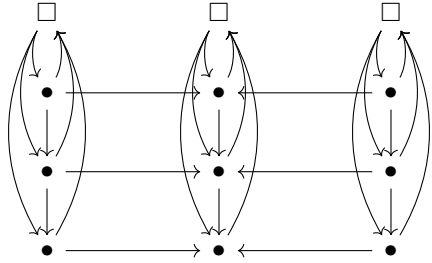
\[\begin{tikzcd}
	\square && \square && \square \\
	\bullet && \bullet && \bullet \\
	\bullet && \bullet && \bullet \\
	\bullet && \bullet && \bullet
	\arrow[curve={height=6pt}, from=1-1, to=2-1]
	\arrow[curve={height=12pt}, from=1-1, to=3-1]
	\arrow[curve={height=18pt}, from=1-1, to=4-1]
	\arrow[curve={height=6pt}, from=1-3, to=2-3]
	\arrow[curve={height=12pt}, from=1-3, to=3-3]
	\arrow[curve={height=18pt}, from=1-3, to=4-3]
	\arrow[curve={height=6pt}, from=1-5, to=2-5]
	\arrow[curve={height=12pt}, from=1-5, to=3-5]
	\arrow[curve={height=18pt}, from=1-5, to=4-5]
	\arrow[curve={height=6pt}, from=2-1, to=1-1]
	\arrow[from=2-1, to=2-3]
	\arrow[from=2-1, to=3-1]
	\arrow[curve={height=6pt}, from=2-3, to=1-3]
	\arrow[from=2-3, to=3-3]
	\arrow[curve={height=6pt}, from=2-5, to=1-5]
	\arrow[from=2-5, to=2-3]
	\arrow[from=2-5, to=3-5]
	\arrow[curve={height=12pt}, from=3-1, to=1-1]
	\arrow[from=3-1, to=3-3]
	\arrow[from=3-1, to=4-1]
	\arrow[curve={height=12pt}, from=3-3, to=1-3]
	\arrow[from=3-3, to=4-3]
	\arrow[curve={height=12pt}, from=3-5, to=1-5]
	\arrow[from=3-5, to=3-3]
	\arrow[from=3-5, to=4-5]
	\arrow[curve={height=18pt}, from=4-1, to=1-1]
	\arrow[from=4-1, to=4-3]
	\arrow[curve={height=18pt}, from=4-3, to=1-3]
	\arrow[curve={height=18pt}, from=4-5, to=1-5]
	\arrow[from=4-5, to=4-3]
\end{tikzcd}\]
and the quiver $\framedrep{\double{Q}}{2}$ is

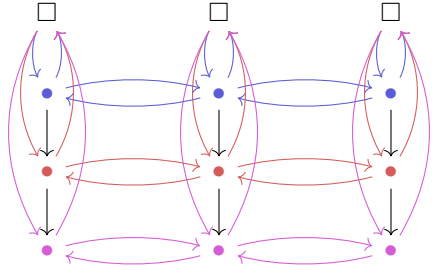
\[\begin{tikzcd}
	\square && \square && \square \\
	\textcolor{rgb,255:red,92;green,92;blue,214}{\bullet} && \textcolor{rgb,255:red,92;green,92;blue,214}{\bullet} && \textcolor{rgb,255:red,92;green,92;blue,214}{\bullet} \\
	\textcolor{rgb,255:red,214;green,92;blue,92}{\bullet} && \textcolor{rgb,255:red,214;green,92;blue,92}{\bullet} && \textcolor{rgb,255:red,214;green,92;blue,92}{\bullet} \\
	\textcolor{rgb,255:red,214;green,92;blue,214}{\bullet} && \textcolor{rgb,255:red,214;green,92;blue,214}{\bullet} && \textcolor{rgb,255:red,214;green,92;blue,214}{\bullet}
	\arrow[color={rgb,255:red,92;green,92;blue,214}, curve={height=6pt}, from=1-1, to=2-1]
	\arrow[color={rgb,255:red,214;green,92;blue,92}, curve={height=12pt}, from=1-1, to=3-1]
	\arrow[color={rgb,255:red,214;green,92;blue,214}, curve={height=18pt}, from=1-1, to=4-1]
	\arrow[color={rgb,255:red,92;green,92;blue,214}, curve={height=6pt}, from=1-3, to=2-3]
	\arrow[color={rgb,255:red,214;green,92;blue,92}, curve={height=12pt}, from=1-3, to=3-3]
	\arrow[color={rgb,255:red,214;green,92;blue,214}, curve={height=18pt}, from=1-3, to=4-3]
	\arrow[color={rgb,255:red,92;green,92;blue,214}, curve={height=6pt}, from=1-5, to=2-5]
	\arrow[color={rgb,255:red,214;green,92;blue,92}, curve={height=12pt}, from=1-5, to=3-5]
	\arrow[color={rgb,255:red,214;green,92;blue,214}, curve={height=18pt}, from=1-5, to=4-5]
	\arrow[color={rgb,255:red,92;green,92;blue,214}, curve={height=6pt}, from=2-1, to=1-1]
	\arrow[color={rgb,255:red,92;green,92;blue,214}, curve={height=-6pt}, from=2-1, to=2-3]
	\arrow[from=2-1, to=3-1]
	\arrow[color={rgb,255:red,92;green,92;blue,214}, curve={height=6pt}, from=2-3, to=1-3]
	\arrow[color={rgb,255:red,92;green,92;blue,214}, curve={height=-6pt}, from=2-3, to=2-1]
	\arrow[color={rgb,255:red,92;green,92;blue,214}, curve={height=-6pt}, from=2-3, to=2-5]
	\arrow[from=2-3, to=3-3]
	\arrow[color={rgb,255:red,92;green,92;blue,214}, curve={height=6pt}, from=2-5, to=1-5]
	\arrow[color={rgb,255:red,92;green,92;blue,214}, curve={height=-6pt}, from=2-5, to=2-3]
	\arrow[from=2-5, to=3-5]
	\arrow[color={rgb,255:red,214;green,92;blue,92}, curve={height=12pt}, from=3-1, to=1-1]
	\arrow[color={rgb,255:red,214;green,92;blue,92}, curve={height=-6pt}, from=3-1, to=3-3]
	\arrow[from=3-1, to=4-1]
	\arrow[color={rgb,255:red,214;green,92;blue,92}, curve={height=12pt}, from=3-3, to=1-3]
	\arrow[color={rgb,255:red,214;green,92;blue,92}, curve={height=-6pt}, from=3-3, to=3-1]
	\arrow[color={rgb,255:red,214;green,92;blue,92}, curve={height=-6pt}, from=3-3, to=3-5]
	\arrow[from=3-3, to=4-3]
	\arrow[color={rgb,255:red,214;green,92;blue,92}, curve={height=12pt}, from=3-5, to=1-5]
	\arrow[color={rgb,255:red,214;green,92;blue,92}, curve={height=-6pt}, from=3-5, to=3-3]
	\arrow[from=3-5, to=4-5]
	\arrow[color={rgb,255:red,214;green,92;blue,214}, curve={height=18pt}, from=4-1, to=1-1]
	\arrow[color={rgb,255:red,214;green,92;blue,214}, curve={height=-6pt}, from=4-1, to=4-3]
	\arrow[color={rgb,255:red,214;green,92;blue,214}, curve={height=18pt}, from=4-3, to=1-3]
	\arrow[color={rgb,255:red,214;green,92;blue,214}, curve={height=-6pt}, from=4-3, to=4-1]
	\arrow[color={rgb,255:red,214;green,92;blue,214}, curve={height=-6pt}, from=4-3, to=4-5]
	\arrow[color={rgb,255:red,214;green,92;blue,214}, curve={height=18pt}, from=4-5, to=1-5]
	\arrow[color={rgb,255:red,214;green,92;blue,214}, curve={height=-6pt}, from=4-5, to=4-3]
\end{tikzcd}\]
Note that for each of the colors blue, red, and purple, the arrows of the corresponding color form a copy of the quiver $\double{\framed{Q}}$. 
\end{example}

A dimension vector for the quiver $\framedrep{\double{Q}}{\ell}$ consists of $\ell+1$ dimension vectors $\dimv{0}, \dots, \dimv{\ell}$, together with a framing vector $\framvec$. Note that an element of $\Rep_{\framedrep{\double{Q}}{\ell}}\left(\dimv{0}, \dots, \dimv{\ell}; \framvec\right)$ consists of:
\begin{itemize}
\item an element $(A^{(m)}, B^{(m)}, i^{(m)}, j^{(m)}) \in \Rep_{\double{\framed{Q}}}\left( \dimv{m}; \framvec\right)$ for every $m = 0, \dots, \ell$, and 
\item maps $\pi_{v}^{(m)}: D^{(m)}_v \to D_v^{(m+1)}$ for every $m = 0, \dots, \ell -1$ and $v \in Q_0$.
\end{itemize}

\begin{definition}
The set $\Rep^0_{\framedrep{\double{Q}}{\ell}}\left(\dimv{0}, \dots, \dimv{\ell}; \framvec\right) \subseteq \Rep_{\framedrep{\double{Q}}{\ell}}\left(\dimv{0}, \dots, \dimv{\ell}; \framvec\right)$ consists of those elements $\left(A^{(\ast)}, B^{(\ast)}, i^{(\ast)}, j^{(\ast)}, \pi^{(\ast)}\right)$ such that the following two conditions are satisfied.

\begin{enumerate}
    \item For every arrow $e: v \to w$  of $Q$ and for every $m = 0, \dots, \ell -1$, the following diagrams commute.
    \begin{equation}\label{eq:compatibility-1}
    \begin{tikzcd}
        D_v^{(m)} \arrow{r}{A^{(m)}_{e}} \arrow[d,"\pi_{v}^{(m)}"'] & D_w^{(m)} \arrow{d}{\pi_{w}^{(m)}} \\  D_v^{(m+1)} \arrow{r}{A^{(m+1)}_{e}} & D_w^{(m+1)}
    \end{tikzcd} \qquad \qquad 
    \begin{tikzcd}
        D_w^{(m)} \arrow{r}{B^{(m)}_{e}} \arrow[d,"\pi_{w}^{(m)}"'] & D_v^{(m)} \arrow[d, "\pi_{v}^{(m)}"] \\  D_w^{(m+1)} \arrow{r}{B^{(m+1)}_{e}} & D_v^{(m+1)}
    \end{tikzcd}
    \end{equation}

    \item For every vertex $v$ of $Q$ and every $m = 0, \dots, \ell-1$, the following diagrams commute.
    \begin{equation}\label{eq:compatibility-2}
    \begin{tikzcd}
 & F_v \arrow[dl,"i_v^{(m)}"'] \arrow{dr}{i_{v}^{(m+1)}} & \\ D_v^{(m)} \arrow{rr}{\pi_{v}^{(m)}} & & D_v^{(m+1)}
 \end{tikzcd} \qquad \qquad  \begin{tikzcd}
 & F_v  & \\ D_v^{(m)} \arrow{ur}{j_v^{(m)}} \arrow{rr}{\pi_{v}^{(m)}} & & D_v^{(m+1)} \arrow[ul, "j_v^{(m+1)}"'].
 \end{tikzcd} 
    \end{equation}
\end{enumerate}
Note that $\Rep^{0}_{\framedrep{\double{Q}}{\ell}}\left(\dimv{0}, \dots, \dimv{\ell}; \framvec\right)$ is closed in $\Rep_{\framedrep{\double{Q}}{\ell}}\left(\dimv{0}, \dots, \dimv{\ell}; \framvec\right)$.
\end{definition}

\begin{remark}
Note that for every $m = 0, \dots, \ell$, the elements $\left(A^{(m)}, B^{(m)}, i^{(m)}, j^{(m)}\right)$ define a representation
\[
D^{(m)} \oplus F = \bigoplus_{v \in Q_0} D^{(m)}_v \oplus \bigoplus_{v \in Q_0} F_v
\]
of the algebra $\C\double{\framed{Q}}$, cf. Remark \ref{rmk:representations}. From this point of view, \eqref{eq:compatibility-1} and \eqref{eq:compatibility-2} simply say that the map
\[
\bigoplus_{v \in Q_0} \pi_v^{(m)} \oplus \bigoplus_{v \in Q_0} \mathrm{id}_{F_v}: D^{(m)} \oplus F \to D^{(m+1)} \oplus F
\]
is a morphism of $\C\double{\framed{Q}}$-modules.
\end{remark}

The following lemma follows easily from the definitions, but will be crucial in what follows.

\begin{lemma}\label{lem:surjective}
    Let $\left(A^{(\ast)}, B^{(\ast)}, i^{(\ast)}, j^{(\ast)}, \pi^{(\ast)}\right) \in \Rep^0_{\framedrep{\double{Q}}{\ell}}\left(\dimv{0}, \dots, \dimv{\ell}; \framvec\right)$, and assume that the element $\left(A^{(m+1)}, B^{(m+1)}, i^{(m+1)}, j^{(m+1)}\right) \in \Rep_{\double{\framed{Q}}}\left(\dimv{m+1}; \framvec\right)$ is stable. Then $\pi_{v}^{(m)}: D^{(m)}_v \to D^{(m+1)}_v$ is surjective for every $v \in Q_0$. 
\end{lemma}
\begin{proof}
Let $D'_v := \pi_{v}^{(m)}\left(D^{(m)}_v\right) \subseteq D^{(m+1)}_v$. By the commutativity of the diagrams \eqref{eq:compatibility-1}, the subspaces $D'_v$ are closed under the action of $A_{e}^{(m+1)}, B_{e}^{(m+1)}$. By the commutativity of the left diagram in \eqref{eq:compatibility-2}, $D'_v$ contains $i_v^{(m+1)}(F_v)$. By stability, it follows that $D'_v = D^{(m+1)}_v$ for every $v \in Q_0$. 
\end{proof}

\begin{definition}\label{def:kernel}
    Let $\left(A^{(\ast)}, B^{(\ast)}, i^{(\ast)}, j^{(\ast)}, \pi^{(\ast)}\right) \in \Rep^0_{\framedrep{\double{Q}}{\ell}}\left(\dimv{0}, \dots, \dimv{\ell}; \framvec\right)$. For each $v \in Q_0$, let us denote $K_v:= \ker\left(\pi_v^{(\ell-1)} \circ \cdots \circ \pi_v^{(0)}\right) \subseteq D_v^{(0)}$, and $K = \bigoplus_{v \in Q_0} K_v$ Note that 
\[
K \oplus 0 \subseteq D^{(0)} \oplus F
\]
is a $\C\double{\framed{Q}}$-submodule. Moreover, if $\left(A^{(\ell)}, B^{(\ell)}, i^{(\ell)}, j^{(\ell)}\right)$ is stable then $D^{(\ell)} \oplus F \cong \left(D^{(0)} \oplus F\right)/(K \oplus 0)$.
\end{definition}

\begin{definition}\label{def:prequotient}
Let $Q$ be a quiver, $\dimv{0}, \dots, \dimv{\ell} \in \Z_{\geq 0}^{Q_0}$ dimension vectors, and $\framvec \in \Z_{\geq 0}^{Q_0}$ a framing vector. We define the pre-quotient parabolic variety $\prequot{Q}{\dimv{0}, \dots, \dimv{\ell}; \framvec}$ to be the set of elements $\left(A^{(\ast)}, B^{(\ast)}, i^{(\ast)}, j^{(\ast)}, \pi^{(\ast)}\right) \in \Rep^0_{\framedrep{\double{Q}}{\ell}}\left(\dimv{0}, \dots, \dimv{\ell}; \framvec\right)$ satisfying the following conditions:
\begin{enumerate}
    \item For every $m = 0, \dots, \ell$, $\left(A^{(m)}, B^{(m)}, i^{(m)}, j^{(m)}\right) \in \Rep_{\double{\framed{Q}}}( \dimv{m}; \framvec)$ is stable. 
    \item For every $m = 0, \dots, \ell$, $\left(A^{(m)}, B^{(m)}, i^{(m)}, j^{(m)}\right)$ satisfies the moment map relations \eqref{eq:moment-map-vertex}. 
    \item The representation $K \oplus 0$ of $\double{\framed{Q}}$ defined in Definition \ref{def:kernel} is trivial, i.e., all paths of positive length act by zero on it. 
\end{enumerate}
\end{definition}

\begin{remark}
    We remark that Condition (3) of Definition \ref{def:prequotient} can be restated as
    \[
    K_v \subseteq \bigcap_{\substack{e \in Q_1 \\ \source(e) = v}} \ker\left(A_e^{(0)}\right) \cap \bigcap_{\substack{e \in Q_1 \\ \target(e) = v}} \ker\left(B_e^{(0)}\right) \cap \ker\left(j^{(0)}_v\right) \qquad \text{for every} \; v \in Q_0.
    \]
\end{remark}

Consider the group $G = G_{\dimv{0}} \times \cdots \times G_{\dimv{\ell}}$ 
Note that:

\begin{itemize}
\item The group $G$ acts naturally on $\Rep_{\framedrep{\double{Q}}{\ell}}\left( \dimv{0}, \dots, \dimv{\ell}; \framvec\right)$, as follows:
\[
(g_{0}, \dots, g_{\ell})\cdot \left(A^{(\ast)}, B^{(\ast)}, i^{(\ast)}, j^{(\ast)}, \pi^{(\ast)}\right) = \left(g_{\ast}\cdot \left(A^{(\ast)}, B^{(\ast)}, i^{(\ast)}, j^{(\ast)}\right), \left(g_{{\ast +1}, v}\pi_v^{(\ast)}g_{\ast,v}^{-1}\right)_{v}\right),
\]
where $g_m = (g_{m,v})_{v \in Q_0} \in G_{\dimv{m}}$. 
\item This action preserves the closed subset $\Rep^{0}$ of elements satisfying \eqref{eq:compatibility-1} and \eqref{eq:compatibility-2}. 
\item The action preserves the variety $\prequot{Q}{\dimv{0}, \dots, \dimv{\ell}; \framvec}$. 
\end{itemize}

\begin{lemma}
   The group $G$ acts freely on $\prequot{Q}{\dimv{0}, \dots, \dimv{\ell}; \framvec}$.
\end{lemma}
\begin{proof}
Since for each $m$ the tuple $\left(A^{(m)}, B^{(m)}, i^{(m)}, j^{(m)}\right) \in \Rep_{\double{\framed{Q}}}\left( \dimv{m}; \framvec\right)$ is stable, this follows directly from \cite[Theorem 5.2.2]{ginzburg}. 
\end{proof}

\begin{definition}\label{def: parabolic-variety-2}
Let $Q$ be a quiver, $\dimv{0}, \dots, \dimv{\ell}$ dimension vectors, and $\framvec$ a framing vector. Let $G := \prod_{m = 1}^{\ell} G_{\dimv{m}}$. We define the \emph{split parabolic quiver variety} as
\[
\pnak{Q}{\dimv{0}, \dots, \dimv{\ell}; \framvec} := \prequot{Q}{\dimv{0}, \dots, \dimv{\ell}; \framvec}/G. 
\]
\end{definition}

We note that, by Lemma \ref{lem:surjective}, if $\pnak{Q}{\dimv{1}, \dots, \dimv{\ell};\framvec}$   is nonempty then we must have $d_v^{(m)} - d_v^{(m+1)} \geq 0$ for every $v \in Q_0$ and every $m = 0, \dots, \ell -1$. 
In other words, we must have
\[
\dimv{m} - \dimv{m+1} \in \Z_{\geq 0}^{Q_0}.
\]

Let us now show that, in fact, the variety $\parnak{\uk}{Q}{\dimvec;\framvec}$ from Definition \ref{def: parabolic-variety-1} is a special case of the varieties we have defined in Definition \ref{def: parabolic-variety-2}. Recall that for the variety $\parnak{\uk}{Q}{\dimvec; \framvec}$ to be defined, we need that $d_v - k_v \geq 0$ for every $v \in Q_0$. We set
\[
\widehat{\dimvec} := \dimvec - \uk \in \Z_{\geq 0}^{Q_0}. 
\]
Now let $\ell := |\uk| = \sum_{v \in Q_0} k_v$, and we let $\dimv{0}, \dots, \dimv{\ell}$ be any sequence of dimension vectors satisfying:
\begin{enumerate}
    \item[(1)] $\dimv{0} = \dimvec$, $\dimv{\ell} = \widehat{\dimvec}$, and  
    \item[(2)] for any $m = 0, \dots, \ell -1$, there exists a vertex $v \in Q_0$ such that $\dimv{m} - \dimv{m+1} = \can_v$. 
\end{enumerate}

\begin{proposition}\label{prop: isomorphism}
Let $\dimvec, \uk \in \Z_{\geq 0}$, and assume $\widehat{\dimvec} = \dimvec - \uk \in \Z_{\geq 0}$. Consider a sequence of dimension vectors $\dimv{0}, \dots, \dimv{\ell}$ as above. Then,
\[
\parnak{\uk}{Q}{\dimvec; \framvec} \cong \pnak{Q}{\dimv{0}, \dots, \dimv{\ell}; \framvec}.
\]
\end{proposition}
\begin{proof}
First, we build a map $\prequot{Q}{\dimv{0}, \dots, \dimv{\ell}; \framvec} \to \Rep^{\uk}_{{\double{\framed{Q}}}}\left(\dimvec; \framvec\right)^{\st}_{0}$. We set $D_v := D_v^{(0)}$, $\widehat{D}_v := D_v^{(\ell)}$, and consider $\pi_v := \pi_{v}^{(\ell-1)} \circ \cdots \circ \pi_v^{(0)}: D_v \to \widehat{D}_v$. By Lemma \ref{lem:surjective}, $\pi_v$ is surjective. Let $K_v := \ker(\pi_v)$ so that $\dim(K_v) = k_v$. By definition, $K_v \subseteq \ker\left(A^{(0)}_e\right)$ for every arrow $e$ such that $\source(e) = v$. Thus, $A^{(0)}_{e}: D_v \to D_{\target(e)}$ factors through a map $\widehat{A}_e: \widehat{D}_v \to D_{\target(e)}$ and similarly for maps $B_{e}$ and $j_v$. This defines a map $\prequot{Q}{\dimv{0}, \dots, \dimv{\ell}; \framvec} \to \Rep^{\uk}_{\double{\framed{Q}}}\left(\dimvec; \framvec\right)$. That the image is contained in $\Rep^{\uk}_{\double{\framed{Q}}}\left(\dimvec; \framvec\right)^{\st}_{0}$ follows since $\left(A^{(0)}, B^{(0)}, i^{(0)}, j^{(0)}\right) \in \mu^{-1}(0)^{\st}$. 

Conversely, let us construct a map $\Rep^{\uk}_{\double{\framed{Q}}}\left(\dimvec; \framvec\right)^{\st}_{0} \to \prequot{Q}{\dimv{0}, \dots, \dimv{\ell}; \framvec}$. To this end, we choose an ordered basis $(\nu_{i,v})_{i = 1, \dots, k_v}$ of the vector space $K_v$ for each $v \in Q_0$, so that $P_v \subseteq \GL(D_v)$ is the subgroup of maps $g_v: D_v \to D_v$ satisfying $g_v(\nu_{i, v}) \in \spann\left\{\nu_{j,v} \mid j \leq i\right\}$.  We define the vector spaces $D^{(m)}_{v}$ as follows. Let $\uk^{(m)} := \dimv{0} - \dimv{m}$, so that $\uk^{(0)} = 0$ and $\uk^{(m)} = \widehat{\dimvec}$. Now, define $K_v^{(m)} = \spann\left\{\nu_{i,v} \mid i \leq k_v^{(m)}\right\}$ and $D_v^{(m)} := D_v/K_v^{(m)}$. By definition, the dimension vector of $\left(D_{v}^{(m)}\right)_{v\in Q_0}$ is precisely $\dimv{m}$. In particular, we have natural maps $\pi_v^{(m)}: D_v^{(m)} \to D_v^{(m+1)}$.

Given an element $\left(\widehat{A}, \widehat{B}, i, \widehat{j}\right) \in \Rep^{\uk}_{\double{\framed{Q}}}\left(\dimvec; \framvec\right)^{\st}_{0}$, it remains to define an element $\left(A^{(m)}, B^{(m)}, i^{(m)}, j^{(m)}\right)$ for every $m = 0, \dots, \ell$. Recall the embedding $\iota: \Rep^{\uk}_{\double{\framed{Q}}}(\dimvec; \framvec) \to \Rep_{\double{\framed{Q}}}(\dimvec; \framvec)$ from Lemma \ref{lem: geometric embedding}. We define $\left(A^{(0)}, B^{(0)}, i^{(0)}, j^{(0)}\right) := \iota\left(\widehat{A}, \widehat{B}, i, \widehat{j}\right)$. To define $\left(A^{(m)}, B^{(m)}, i^{(m)}, j^{(m)}\right)$ for higher $m$, we note that we have projections:
\[
D_v \twoheadrightarrow D_v^{(m)} \twoheadrightarrow \widehat{D}_v
\]
so that, post-composing with the first projection, the element $\left(\widehat{A}, \widehat{B}, i, \widehat{j}\right)$ defines a representation in $\Rep_{\double{\framed{Q}}}^{\uk - \uk^{(m)}}(\dimvec - \uk^{(m)}; \framvec)$. Thus, applying $\iota$ from Lemma \ref{lem: geometric embedding} again, this yields an element $\left(A^{(m)}, B^{(m)}, i^{(m)}, j^{(m)}\right)$. By definition, the diagrams \eqref{eq:compatibility-1} and \eqref{eq:compatibility-2} commute. We still need to verify that $\left(A^{(m)}, B^{(m)}, i^{(m)}, j^{(m)}\right)$ is stable and satisfies the moment map equations. By assumption, both of these properties hold when $m = 0$. Now, the moment map equations clearly hold on $K^{(m)}$, since all maps $A^{(0)}, B^{(0)}$ and $j^{(0)}$ act by zero on it, hence they hold on $D^{(m)} = D^{(0)}/K^{(m)}$. 

In particular, stability follows from Lemma \ref{lem:stable-quotient} below.  This Lemma is of independent interest and also independent of the intervening material. We have then constructed a map
\[
\Phi: \Rep^{\uk}_{\double{\framed{Q}}}\left(\dimvec; \framvec\right) \to \prequot{Q}{\dimv{0}, \dots, \dimv{\ell}; \framvec}.
\]
Since the group $P_{\uk}$ preserves $\spann\left\{\nu_{i,v} \mid i \leq k_v^{(m)}\right\}$, we have an induced map $P_{\uk} \to G = \prod_{m = 0}^{\ell} G_{\dimv{m}}$ that intertwines the $P_{\uk}$-action on the source of $\Phi$, with the $G$-action on the target. This gives us the desired map
\[
\Phi: \parnak{\uk}{Q}{\dimvec; \framvec} \to \pnak{Q}{\dimv{0}, \dots, \dimv{\ell}; \framvec}. 
\]
Note that if $\left(A^{(\ast)}, B^{(\ast)}, i^{(\ast)}, j^{(\ast)}, \pi^{(\ast)}\right) \in \prequot{Q}{\dimv{0}, \dots, \dimv{\ell}; \framvec}$ then, up to the $G$-action, we may assume that the map $\pi_v^{(m)} \circ \cdots \circ \pi_v^{(0)}: D_v^{(0)} \to D_v^{(0)}/K_v^{(m)}$ is the natural projection, and this implies that $\Phi$ is surjective. 

For injectivity, if $\Phi\left(\widehat{A}, \widehat{B}, i, \widehat{j}\right)$ and $\Phi(\widehat{A}', \widehat{B}', i', \widehat{j}')$ are in the same $G$-orbit then an element in $G_{\dimv{0}}$ relating them must preserve the subspaces $\spann\left\{\nu_{i, v} \mid i \leq k^{(m)}_v\right\}$, which implies that the element belongs to $P_{\uk}$. The result follows. 
\end{proof}

\begin{lemma}\label{lem:stable-quotient}
    Let $D^{(1)} = (A, B, i, j) \in \Rep_{\double{\framed{Q}}}\left( \dimv{1}; \framvec\right)$ be a stable representation, and $D^{(2)} = (A', B', i', j') \in \Rep_{\double{\framed{Q}}}\left( \dimv{2}; \framvec\right)$ a representation that is \textbf{not} assumed to be stable. Assume there exist surjective maps $\pi_v: D_v^{(1)}  \to D_v^{(2)} $ that make the diagrams \eqref{eq:compatibility-1} and \eqref{eq:compatibility-2} commute. Then, $(A', B', i', j')$ is stable. 
\end{lemma}
\begin{proof}
    Let $\left(S'_v \subseteq D_{v}^{(2)}\right)_{v \in Q_0}$ be a collection of subspaces containing $i'(F_v)$ and $(A', B')$-stable. We must show that $S'_v = D_v^{(2)}$ for every $v \in Q_0$. For this, consider $S_v := \pi_v^{-1}(S'_v) \subseteq D_v^{(1)}$. 

    Since $S'_v$ contains $i'(F_v)=\pi_vi_v(F_v)$, we have $i_v(F_v)\subset S_v$. Similarly, for any arrow $e:v\to w$ we have $\pi_wA_e(S_v)=A_e\pi_v(S_v)=A'_e(S'_v)\subset S'_w$, so $A_e(S_v)\subset S_w$.
  We conclude that $S_v$ contains $i(F_v)$ and is $(A,B)$-stable, so $S_v = D_v^{(1)}$ by stability of $(A, B, i, j)$. Thus, by surjectivity of $\pi$,
  $$S'_v = \pi_v(S_v) = \pi_v\left(D_v^{(1)}\right) = D_v^{(2)}.$$
\end{proof}

\begin{remark}
    Proposition \ref{prop: isomorphism} generalizes as follows. The variety $\pnak{Q}{\dimv{0}, \dots, \dimv{\ell}; \framvec}$ can be realized as the quotient of $\Rep_{Q}^{\dimv{0} - \dimv{\ell}}\left(\dimv{0}; \framvec\right)^{\st}_{0}$ by the action of a parabolic subgroup $P$ of $\GL_{\dimv{0}}$, which may be constructed as follows. Consider the difference vectors
    \[
    \diff{i} := \dimv{i} - \dimv{i+1}.
    \]
    Then, $P_v$ consists of those $(\dimv{0}_v - \dimv{\ell}_v, \dimv{\ell}_v)$-block upper triangular matrices, whose upper-left part is block upper triangular with blocks of size $\diff{i}_v$, $i = 0, \dots, \ell - 1$. In particular, $\pnak{Q}{\dimv{0}, \dots, \dimv{\ell}; \framvec}$ depends only on $\framvec, \dimv{0}$ and the unordered sequence of difference vectors $\diff{i}$. Moreover, one can always assume that $\diff{i}$ is concentrated at a single vertex.
\end{remark}

\begin{remark}\label{rmk: stability-2}
    As mentioned in Remark \ref{rmk: stability} above, it would be interesting to define and study split parabolic quiver varieties for a more general stability condition. This is done in \cite{NakajimaYoshioka1} for the variety $\pnak{Q}{d, d-k; f}$ when $Q$ is the Jordan quiver. The generalization of $\pnak{Q}{d, d-k; f}$ depends on a character of $\GL(d) \times \GL(d-k)$ (see the paragraph after Theorem 1.1 in \cite{NakajimaYoshioka1}). Considering different pairs leads to interesting wall-crossing phenomena, hence it would be worthwhile to generalize this to arbitrary quivers. 
\end{remark}

\subsection{Smoothness and dimension}

The following result is analogous to \cite[Corollary 3.12]{nakajima-km}.

\begin{theorem}
\label{thm: smooth}
Let $\left(\widehat{X}, \widehat{Y}, a, \widehat{b}\right) \in \Rep_{Q}^{\uk}\left(\dimvec; \framvec\right)^{\st}_{0}$. Consider the following maps
\begin{multline}\label{eq: tangent-space}
\mathrm{Lie}(P_{\uk}) \xrightarrow{\jmath} \bigoplus_{e \in Q_1} \Hom(\widehat{D}_{\source(e)}, D_{\target(e)}) \oplus \Hom(\widehat{D}_{\target(e)}, D_{\source(e)}) \oplus \bigoplus_{v \in Q_0} \Hom(F_v, D_v) \oplus \Hom(\widehat{D}_v, F_v) \\ \xrightarrow{d\mu}  \bigoplus_{v \in Q_0} \Hom(\widehat{D}_v, D_v).
\end{multline}
given by
\begin{align*}
\jmath
\left((\xi_{v})_{v \in Q_0}\right) = \left(\xi_{\target(e)}\widehat{X}_e - \widehat{X}_e\hat{\xi}_{\source(e)}, \xi_{\source(e)}\widehat{Y}_e - \widehat{Y}_e\widehat{\xi}_{\target(e)}, \xi_va_v, b_v\widehat{\xi}_v\right), \\ 
d\mu\left(\widehat{A}, \widehat{B}, i, \widehat{j}\right) = \left(\sum_{\substack{e \in Q_1 \\ \target(e) = v}} \left[\widehat{A}_e\pi_{\source(e)}\widehat{Y}_e + \widehat{X}_e\pi_{\source(e)}\widehat{B}_e\right] -\sum_{\substack{e \in Q_1 \\ \source(e) = v}} \left[\widehat{Y}_e\pi_{\target(e)}\widehat{A}_e + \widehat{B}_e\pi_{\target(e)}\widehat{X}_e\right] + i\widehat{b} + a\widehat{j} \right)_{v \in Q_0}
\end{align*}
where, for $\xi_v \in \mathrm{Lie}(P_{\uk})_{v} \subseteq \End(D_v, D_v)$, $\widehat{\xi}_v$ denotes its projection to $\End(\widehat{D}_v, \widehat{D}_v)$. Then,
\begin{enumerate}
\item $\jmath$ is injective and $d\mu$ surjective,
\item $d\mu \circ \jmath = 0$, and 
\item the tangent space to $\parnak{\uk}{Q}{\dimvec; \framvec}$ at the $P_{\uk}$-orbit of $(\widehat{X}, \widehat{Y}, a, b)$ is the homology of the complex given in \eqref{eq: tangent-space}. 
\end{enumerate}
\end{theorem}
\begin{proof}
Recall the embedding \eqref{eq: geometric embedding} from Lemma \ref{lem: geometric embedding}. Let us denote $\iota\left(\widehat{X}, \widehat{Y}, a, \widehat{b}\right) =: (X,Y,a,b)$. Similarly, for $\left(\widehat{A}, \widehat{B}, i, \widehat{j}\right) \in \Rep^{\uk}_{Q}(\dimvec; \framvec)$, we denote $\iota\left(\widehat{A}, \widehat{B}, i, \widehat{j}\right) =: (A, B, i, j)$. We have the following commutative diagram, where the bottom row is defined using $(X,Y,i,j)$ as in \cite[Corollary 3.12]{nakajima-km} and the right vertical map is given by $M \mapsto M\pi$.
\begin{center}
\begin{tikzcd}
\mathrm{Lie}(P_{\uk}) \arrow[r, "\jmath"] \arrow[d, hook] & \Rep^{\uk}_{Q}\left(\dimvec; \framvec\right) \arrow[r, "d\mu"] \arrow[d, hook, "\iota"] & \bigoplus_{v} \Hom(\widehat{D}_v, D_v) \arrow[d, hook] \\
\mathrm{Lie}(G_{\dimvec}) \arrow[r, "\jmath"] & \Rep_{\double{\framed{Q}}}\left(\dimvec; \framvec\right) \arrow[r, "d\mu"] & \bigoplus_{v} \End(D_v)
\end{tikzcd}
\end{center}
By \cite[Corollary 3.12]{nakajima-km}, in the bottom row we have that $\jmath$ is injective, $d\mu$ is surjective, and $d\mu\circ \jmath = 0$. So the corresponding results also hold for the top row except, perhaps, for the surjectivity of $d\mu$. 

For the surjectivity of $d\mu$,  we compute the orthogonal complement $d\mu^{\perp}$ for the image of $d\mu$. More precisely, consider a collection of matrices $C_v\in \Hom(D_v,\widehat{D}_v)$ and the trace pairing
$$
\Tr\left(d\mu\left(\widehat{A}, \widehat{B}, i, \widehat{j}\right)\cdot C\right)=\sum_{v}\Tr\left(d\mu\left(\widehat{A}, \widehat{B}, i, \widehat{j}\right)_v C_v\right)=
$$
\begin{multline}
\label{eq: trace pairing}
\sum_{e\in Q_1}\Tr\left(\widehat{A}_e\left[
\pi_{\source(e)}\widehat{Y}_eC_{\target(e)}-C_{\source(e)}\widehat{Y}_e\pi_{\target(e)}
\right]\right)+\\
\sum_{e\in Q_1}\Tr\left(\widehat{B}_e\left[
C_{\target(e)}\widehat{X}_e\pi_{\source(e)}-
\pi_{\target(e)}\widehat{X}_eC_{\source(e)}
\right]\right)+\\
\sum_v\Tr\left(i_v\widehat{b}_vC_v+C_va_v\widehat{j}_v\right).
\end{multline}
The collection $(C_v)$ is orthogonal to the image of $d\mu$ if \eqref{eq: trace pairing} vanishes for all choices of $\widehat{A}_e,\widehat{B}_e,i,\widehat{j}$. This impies the system of equations
\begin{equation}
\label{eq: orthogonal dmu}
\begin{cases}
\pi_{\source(e)}\widehat{Y}_eC_{\target(e)}-C_{\source(e)}\widehat{Y}_e\pi_{\target(e)}=0\\
C_{\target(e)}\widehat{X}_e\pi_{\source(e)}-
\pi_{\target(e)}\widehat{X}_eC_{\source(e)}\\
\widehat{b}_vC_v=0\\
C_va_v=0.
\end{cases}
\end{equation}
We claim that \eqref{eq: orthogonal dmu} implies $C_v=0$, so $d\mu^{\perp}=0$. Indeed, the stability condition implies that $a_v(F_v)$ generate $D$ under the action of $Y_e=\widehat{Y}_{e}\pi_{\target(e)}$ and $X_e=\widehat{X}_{e}\pi_{\source(e)}$. By \eqref{eq: orthogonal dmu} we get $C_va_v=0$. Suppose that $\zeta\in D_{\target(e)}$ and $C_{\target(e)}\zeta=0$. Then by \eqref{eq: orthogonal dmu} we get
$$
C_{\source(e)}Y_{e}(\zeta)=C_{\source(e)}\widehat{Y}_{e}\pi_{\target(e)}(\zeta)=\pi_{\source(e)}\widehat{Y}_eC_{\target(e)}(\zeta)=0.
$$
Similarly, if $\zeta\in D_{\source(e)}$ and $C_{\source(e)}\zeta=0$ then 
$C_{\target(e)}X_e\zeta=0$.

Finally, note that $\jmath$ is the differential of the $P_{\uk}$-action on the space $\Rep^{\uk}_{Q}\left(\dimvec; \framvec\right)$, whereas $d\mu$ is the differential of the restriction of the moment map to $\iota\left(\Rep^{\uk}_{Q}\left(\dimvec; \framvec\right)\right)$ with image in $\bigoplus \Hom(\widehat{D}_v, D_v)$. The claim on the tangent space now follows from the definition of $\parnak{\uk}{Q}{\dimvec; \framvec}$. 
\end{proof}

\begin{corollary}
The space $\parnak{\uk}{Q}{\dimvec; \framvec}$ is either empty or it is a smooth variety of dimension
\begin{multline*}
\dim \parnak{\uk}{Q}{\dimvec; \framvec}= \\
\sum_{e}\left[(d_{\target(e)}-k_{\target(e)})d_{\source(e)}+(d_{\source(e)}-k_{\source(e)})d_{\target(e)}\right] + \sum_{v}(2d_v - k_v)f_v-\sum_{v}\left[2(d_v-k_v)d_v+\frac{k_v(k_v+1)}{2}\right]= \\
\framvec^T(2\dimvec-\uk)-\dimvec^T\Cartan(\dimvec-\uk)-\sum_v\frac{k_v(k_v+1)}{2}=
\dim \nak_{Q}(\dimvec; \framvec)-\framvec^T\uk+\dimvec^T\Cartan\uk - \sum_v\frac{k_v(k_v+1)}{2},
\end{multline*}
where $\Cartan$ is the Cartan matrix in \eqref{eq: def Cartan}.
\end{corollary}

\begin{corollary}
The space $\pnak{Q}{\dimvec,\dimvec-\uk; \framvec}$ is either empty or it is a smooth variety of dimension
\begin{equation}
\label{eq: dim BN correspondence}
\dim \nak_{Q}(\dimvec; \framvec)-\framvec^T\uk+\dimvec^T\Cartan\uk - \sum_v k_v^2.
\end{equation}
\end{corollary}

\begin{proof}
The proof is identical but instead of the parabolic subgroup $P_{\uk}$ (resp. Lie algebra $\Lie(P_{\uk})$) we take the subgroup of $(n-k,k)$ block triangular matrices. 
\end{proof}

\begin{corollary}
More generally, consider a sequence of dimension vectors $\dimvec^{(0)},\ldots,\dimvec^{(\ell)}$ with $\dimvec^{(0)}-\dimvec^{(\ell)}=\uk$. Then the space $\pnak{Q}{\dimvec^{(0)},\ldots,\dimvec^{(\ell)}; \framvec}$ is either empty or it is a smooth variety of dimension
$$
\dim \nak_{Q}(\dimvec; \framvec)-\framvec^T\uk+\dimvec^T\Cartan\uk -\sum_{v}\sum_{1\le m_1\le m_2\le \ell}\left(d^{(m_1)}_v-d^{(m_1-1)}_v\right)
\left(d^{(m_2)}_v-d^{(m_2-1)}_v\right).
$$
\end{corollary}

\subsection{Maps between split parabolic quiver varieties}

In this section, we construct maps between split parabolic quiver varieties that will be important in Section \ref{sec: BN} and in the sequel paper \cite{BqtQ}. First, we compare split parabolic quiver varieties to usual quiver varieties. 

\begin{lemma}\label{lem: quiver-variety-embedding}
We have an embedding
\begin{equation}\label{eq:embedding}
\pnak{Q}{\dimv{0}, \dots, \dimv{\ell}; \framvec} \hookrightarrow \nak_Q\left(\dimv{0}; \framvec\right) \times \cdots \times \nak_Q\left(\dimv{\ell}; \framvec\right).
\end{equation}
\end{lemma}
\begin{proof}
For the purposes of this proof, we will write
\[
\mu^{-1}_m(0)^{\st} := \left\{(A, B, i, j) \in \Rep_{D(Q^\heartsuit)}\left(\dimv{m}; \framvec\right) : (A, B, i, j) \; \text{satisfies} \; \eqref{eq:moment-map-vertex} \; \text{and is stable.} \right\}.
\]
For each $m = 0, \dots, \ell$  we have a  map $\prequot{Q}{\dimv{0}, \dots, \dimv{\ell}; \framvec} \to \mu^{-1}_m(0)^{\st}$ that is equivariant with respect to the obvious projection $G \to G_{\dimv{m}}$. This yields a map
\begin{equation}\label{eq:map-component}
\comp_m: \pnak{Q}{\dimv{0}, \dots, \dimv{\ell}; \framvec} \to \nak_Q\left(\dimv{m}; \framvec\right)
\end{equation}
so that the map in \eqref{eq:embedding} is simply the product $\comp_0 \times \cdots \times \comp_{\ell}$. We have to verify that this map is indeed injective. 

Assume that the $G$-orbits of $\left(A^{(\ast)}, B^{(\ast)}, i^{(\ast)}, j^{(\ast)}, \pi^{(\ast)}\right)$ and $\left(\widetilde{A}^{(\ast)}, \widetilde{B}^{(\ast)}, \widetilde{i}^{(\ast)}, \widetilde{j}^{(\ast)}, \widetilde{\pi}^{(\ast)}\right)$ map to the same element under \eqref{eq:embedding}. Up to the $G$-action, we may assume that $A^{(\ast)} = \widetilde{A}^{(\ast)}$, $B^{(\ast)} = \widetilde{B}^{(\ast)}$, $i^{(\ast)} = \widetilde{i}^{(\ast)}$ and $j^{(\ast)} = \widetilde{j}^{(\ast)}$. We need to show that $\pi^{(\ast)} = \widetilde{\pi}^{(\ast)}$. Fix $m$ and set $D'_v := \ker\left(\pi_{v}^{(m)} - \widetilde{\pi}_{v}^{(m)}\right) \subseteq D_v^{(m)}$. By the left diagram in \eqref{eq:compatibility-2}, $D'_v$ contains the image of $i^{(m)}_v$ and, by \eqref{eq:compatibility-1}, the spaces $D'_v, v \in Q_0$ are closed under the action of $A^{(m)}, B^{(m)}$. The result now follows by stability of $\left(A^{(m)}, B^{(m)}, i^{(m)}, j^{(m)}\right)$.
\end{proof}

Lemma \ref{lem: quiver-variety-embedding} allows us to compare the split parabolic quiver varieties to classical Nakajima correspondences. Let $v \in Q_0$ be a vertex of $Q$, and let $\dimvec$ be a dimension vector. In \cite[Section 5]{nakajima-km}, Nakajima introduced a correspondence $N_Q\left(\dimvec, \dimvec - \can_v; \framvec\right) \subseteq \nak_Q\left(\dimvec; \framvec\right) \times \nak_Q\left(\dimvec - \can_v; \framvec\right)$ consisting of pairs of representations $(D, \widehat{D})$ such that there exists a surjective map $\pi: D \twoheadrightarrow \widehat{D}$. 

\begin{proposition}\label{prop: comparison-nakajima-correspondences}
    We have
    \[
    N_Q\left(\dimvec, \dimvec - \can_v; \framvec\right) \simeq \pnak{Q}{\dimvec, \dimvec - \can_v; \framvec} \times \C^{2|Q_1(v,v)|}. 
    \]
\end{proposition}
\begin{proof}
By definition, the embedding $\pnak{Q}{\dimvec, \dimvec - \can_v; \framvec} \hookrightarrow \nak_Q(\dimvec; \framvec) \times \nak_Q(\dimvec - \can_v; \framvec)$ has its image in $N_Q\left(\dimvec, \dimvec - \can_v; \framvec\right)$. Now let $(D, \widehat{D}) \in N_Q\left(\dimvec, \dimvec - \can_v; \framvec\right)$. The kernel $K$ of the map $\pi: D \twoheadrightarrow \widehat{D}_v$ is a $1$-dimensional representation concentrated at $v$, so the only maps that are possibly nonzero on $K$ are the loops at $v$. If $e$ is a loop at $v$ then, since $K$ is $1$-dimensional, then both $A_e|_K$ and $B_e|_K$ are scalars, say $a_e, b_e$. Replacing $A_e$ by $A_e - a_e$ and $B_e$ by $B_e-b_e$ does not affect the stability condition and, since scalars are central, it does not affect the moment map equations. This gives us an element in the image of $\pnak{Q}{\dimvec, \dimvec - \can_v; \framvec}$ and the result follows.  
\end{proof}

\begin{lemma}\label{lem: maps}
Assume $(A^{(\ast)}, B^{(\ast)}, i^{(\ast)}, j^{(\ast)}, \pi^{(\ast)}) \in \prequot{Q}{\dimv{0}, \dots, \dimv{\ell}; \framvec}$. Then
    \begin{enumerate}
        \item[(a)] For every $v \in Q_0$ and every $m = 0, \dots, \ell-1$, the map $\pi_{v}^{(m)}$ is surjective.
        \item[(b)] Forgetting the tuple $\left(A^{(0)}, B^{(0)}, i^{(0)}, j^{(0)}, \pi^{(0)}\right)$ gives an element of $\prequot{Q}{\dimv{1}, \dots, \dimv{\ell}; \framvec}$. 
        \item[(c)] Forgetting the tuple $\left(A^{(\ell)}, B^{(\ell)}, i^{(\ell)}, j^{(\ell)}, \pi^{(\ell -1)}\right)$ gives an element of $\prequot{Q}{\dimv{0}, \dots, \dimv{\ell-1}; \framvec}$. 
    \end{enumerate}
Moreover, the maps in (b) and (c) descend to maps
\[
\begin{array}{c}
P_{+}:\pnak{Q}{\dimv{0}, \dots, \dimv{\ell}; \framvec} \to \pnak{Q}{\dimv{1}, \dots, \dimv{\ell}; \framvec}, \\
P_{-}:\pnak{Q}{\dimv{0}, \dots, \dimv{\ell}; \framvec} \to \pnak{Q}{\dimv{0}, \dots, \dimv{\ell -1}; \framvec}.
\end{array}
\]
\end{lemma}
\begin{proof}
Statement (a) follows from Lemma \ref{lem:surjective}. For (b) and (c), the only thing that we have to check is that the new elements satisfy condition (3) of Definition \ref{def:prequotient}. 

Let us first verify statement (c). For this, note that $\ker\left(\pi_{v}^{(\ell-2)} \circ \cdots \circ \pi_{v}^{(0)}\right) \subseteq \ker\left(\pi_{v}^{(\ell -1)} \circ \pi_{v}^{(\ell -2)} \cdots \circ \pi_{v}^{(0)}\right)$, so if a map acts by zero on the larger subspace it also does so in the smaller subspace. 

Let us now verify (b). Assume $d \in \ker\left(\pi_{v}^{(\ell-1)} \circ \cdots \circ \pi_{v}^{(1)}\right)$. We must show that $L(d) = 0$, where $L$ is either $A^{(1)}_{e}$,  $B^{(1)}_{e}$, or $j^{(1)}_{v}$. By Lemma \ref{lem:surjective}, we can find $d' \in D_v^{(0)}$ such that $d = \pi_{v,0}(d')$. Note that $d' \in \ker\left(\pi_{v}^{(\ell-1)} \circ \cdots \circ \pi_{v}^{(0)}\right)$. From here, the result follows easily using commutativity of \eqref{eq:compatibility-1} and \eqref{eq:compatibility-2}.
\end{proof}

We will need two important special cases of the maps in Lemma \ref{lem: maps}. First, we consider the case $\ell=1$ with $\dimvec^{(0)}=\dimvec$ and $\dimvec^{(1)}=\dimvec-\uk.$ In this case we get  two projections
\begin{equation}
\label{eq: def p and q}
\begin{tikzcd}
 &\pnak{Q}{\dimvec,\dimvec-\uk; \framvec} \arrow[swap]{dl}{p_{+}} \arrow{dr}{p_{-}} & \\
\nak_{Q}\left(\dimvec-\uk; \framvec\right) & & \nak_{Q}\left(\dimvec; \framvec\right) 
\end{tikzcd}
\end{equation}
which send a pair of representations $\left(D,\widehat{D}\right)$ to $D$ and $\widehat{D}$ respectively. Alternatively, one can think of the pair $(p_-,p_+)$ as a special case of the map \eqref{eq:embedding}. In Section \ref{sec: BN} we study the geometric properties of the maps \eqref{eq: def p and q} in detail, describe their images and prove that these are stratified Grassmannian bundles (see Lemmas \ref{lem: BN projection 1} and \ref{lem: second projection}).

Furthermore, we can define the maps between the spaces $\parnak{\uk}{Q}{\dimvec;\framvec}$ as follows. 

\begin{corollary}
Let $v \in Q_0$, and assume that $k_v > 0$. Then, we have maps
\begin{equation}
\label{eq: def P plus minus}
\begin{tikzcd}
 & \parnak{\uk}{Q}{\dimvec; \framvec} \arrow[swap]{dl}{P_{+,v}} \arrow{dr}{P_{-,v}}& \\
\parnak{\uk - \can_v}{Q}{\dimvec-\can_v; \framvec} & & \parnak{\uk - \can_v}{Q}{\dimvec; \framvec}. 
\end{tikzcd}
\end{equation}
\end{corollary}
\begin{proof}
For the map $P_{+,v}$ we use the identification $\parnak{\uk}{Q}{\dimvec; \framvec}$ with $\pnak{Q}{\dimvec, \dimv{1}, \dots, \dimv{\ell}; \framvec}$. We may assume without loss of generality that $\dimvec - \dimv{1} = \can_v$. Then the map $P_{+,v}$ is obtained as in Lemma \ref{lem: maps} (b). For $P_{-,v}$ now we assume that $\dimv{\ell -1} - \dimv{\ell} = \can_v$, and now the map appears as in Lemma \ref{lem: maps} (c).
\end{proof}

Let us provide down-to-earth descriptions of the maps $P_{+,v}$ and $P_{-,v}$ of \eqref{eq: def P plus minus}. Note that an element of $\parnak{\uk}{Q}{\dimvec; \framvec}$ consists of: 
\begin{itemize}
\item an element $[(A, B, i, j)] \in \nak_{Q}(\dimvec; \framvec)$, and
\item for each vertex $w$, a flag  $0 \subseteq K_{1,w} \subseteq \cdots \subseteq K_{k_w, w} \subseteq \bigcap_{\source(e) = w} \ker(A_e) \cap \bigcap_{\target(e) = w}\ker(B_e) \cap \ker(j_w)$, where $\dim(K_{i,w}) = i$. 
\end{itemize}

Then, the map $P_{-,v}$ simply forgets about the subspace $K_{k_v, v}$ (but leaves the element $[(A, B, i, j)]$ unaffected), while the map $P_{+,v}$ takes the quotient by $K_{1,v}$ (which is allowed by Lemma \ref{lem:stable-quotient}). 

The maps \eqref{eq: def P plus minus} play a central role in our follow-up paper  \cite{BqtQ}, and we give a geometric description of their fibers in Corollary \ref{cor: fibers full flags}.

\subsection{Isotropic property} By Lemma \ref{lem: quiver-variety-embedding} we have an embedding $\pnak{Q}{\dimvec, \dimvec - \uk; \framvec} \hookrightarrow \pnak{Q}{\dimvec; \framvec} \times \pnak{Q}{\dimvec - \uk; \framvec}$, and the variety $\pnak{Q}{\dimvec; \framvec} \times \pnak{Q}{\dimvec - \uk; \framvec}$ is symplectic where, as usual, the symplectic form is negated on the second factor. The goal of this section is to show that $\pnak{Q}{\dimvec, \dimvec-\uk; \framvec}$ is always isotropic in $\pnak{Q}{\dimvec; \framvec} \times \pnak{Q}{\dimvec - \uk; \framvec}$ and, moreover, it is Lagrangian if the support of $\uk$ is totally disconnected as a subgraph of $Q$.  First, we compare $\dim\pnak{Q}{\dimvec, \dimvec - \uk; \framvec}$ to half the dimension of the product $\pnak{Q}{\dimvec; \framvec} \times \pnak{Q}{\dimvec - \uk; \framvec}$.

\begin{lemma}
\label{lem: less than half}
We have 
$$
\dim \pnak{Q}{\dimvec,\dimvec-\uk;\framvec}=\frac{1}{2}\left(\dim \nak_{Q}(\dimvec;\framvec)+\dim \nak_{Q}(\dimvec-\uk;\framvec)\right)-\sum_{v,w}|Q_1(v,w)|k_vk_w.
$$
\end{lemma}

\begin{proof}
One can check (see, for example, the proof of Lemma \ref{lem: second projection dim preimage} below) that
$$
\dim \nak_Q(\dimvec-\uk;\framvec)=\dim \nak_Q(\dimvec;\framvec)+2\uk^T\uk^0[\dimvec]-\uk^T\Cartan\uk,
$$
where $\uk^0[\dimvec]$ is as defined in the introduction \eqref{eq:def-k0-intro}, so
$$
\frac{1}{2}\left(\dim \nak_{Q}(\dimvec;\framvec)+\dim \nak_{Q}(\dimvec-\uk;\framvec)\right)=\dim \nak_Q(\dimvec;\framvec)+\uk^T\uk^0[\dimvec]-\frac{1}{2}\uk^T\Cartan\uk.
$$
On the other hand,
$$
\uk^T\Cartan\uk=\sum_{v,w}(2\delta_{v,w}-|Q_1(v,w)|-|Q_1(w,v)|)k_vk_w=2\sum_v k_v^2-2\sum_{v,w}|Q_1(v,w)|k_vk_w,
$$
so the result follows from \eqref{eq: dim BN correspondence new} below, that is independent of the intervening material.
\end{proof}

Lemma \ref{lem: less than half} shows that the correspondence $\pnak{Q}{\dimvec,\dimvec-\uk;\framvec}$ is at most half-dimensional. In fact, we have the following stronger result.

\begin{proposition}\label{prop: isotropic} 
If nonempty, $\pnak{Q}{\dimvec, \dimvec - \uk; \framvec}$ is isotropic in $\pnak{Q}{\dimvec; \framvec} \times \pnak{Q}{\dimvec - \uk; \framvec}$. Moreover, it is Lagrangian if and only if whenever $k_v, k_w \neq 0$, there are no arrows between $v$ and $w$. 
\end{proposition}
\begin{proof}
We proceed by induction on $|\uk| = \sum k_v$. If $|\uk| = 1$ then there exists $v \in Q_0$ such that $\uk = \can_v$. If there are no loops at the vertex $v$, the result follows from Proposition \ref{prop: comparison-nakajima-correspondences} together with \cite[Theorem 5.7]{nakajima-km}, that claims that the Nakajima correspondence $N_Q\left(\dimvec, \dimvec - \can_v; \framvec\right)$ is Lagrangian. If there are loops at $v$, the same proof as in \cite[Theorem 5.7]{nakajima-km} shows that the Nakajima correspondence $N_Q\left(\dimvec, \dimvec - \can_v; \framvec\right)$ is coisotropic. 
However, we can modify the complex \cite[(5.1)]{nakajima-km} in a similar way to what is done in \eqref{eq: tangent-space} to show that the variety defined analogously to $\pnak{Q}{\dimvec, \dimvec - \can_v; \framvec}$ but only requiring that $B_e$ is zero on $K_v$ for each loop $e: v \to v$, is coisotropic and half-dimensional, thus Lagrangian.
It follows that $\pnak{Q}{\dimvec, \dimvec - \can_v; \framvec}$ is isotropic.

For the inductive step, consider $\uk' = \uk + \can_v$, and consider the product
\[
\pnak{Q}{\dimvec; \framvec} \times \pnak{Q}{\dimvec - \uk; \framvec} \times \pnak{Q}{\dimvec - \uk - \can_v; \framvec}. 
\]
For $i, j \in \{1,2,3\}$, let $p_{ij}$ be the projection to the corresponding factors. We claim that
\begin{equation}\label{eq: composition-containment}
\pnak{Q}{\dimvec, \dimvec - \uk - \can_v; \framvec} \subseteq p_{13}\left(p_{12}^{-1}(\pnak{Q}{\dimvec, \dimvec - \uk; \framvec} \cap p_{23}^{-1}(\pnak{Q}{\dimvec - \uk, \dimvec - \uk - \can_v; \framvec})\right). 
\end{equation}
Note that, if \eqref{eq: composition-containment} holds, then $\pnak{Q}{\dimvec, \dimvec - \uk - \can_v; \framvec}$ is indeed isotropic by \cite[Proposition 2.7.51]{chriss-ginzburg}. To show \eqref{eq: composition-containment}, let $(D, \widehat{D}) \in \pnak{Q}{\dimvec, \dimvec - \uk - \can_v; \framvec}$, so that we have a surjection $\pi: D \to \widehat{D}$ whose kernel $K$ is a trivial representation of dimension vector $\uk + \can_v$ (and framing vector $0$).  We choose a one-dimensional subspace of $K_v$ so that we can factor the surjection $\pi: D \to \widehat{D}$ as $D \twoheadrightarrow \widetilde{D} \twoheadrightarrow \widehat{D}$. By Lemma \ref{lem:stable-quotient}, the representation $\widetilde{D}$ is stable, and $(D, \widehat{D}) = p_{13}(D, \widetilde{D}, \widehat{D})$, so \eqref{eq: composition-containment} follows. 

To prove the last claim, note that $\pnak{Q}{\dimvec, \dimvec - \uk; \framvec}$ is Lagrangian if and only if
$$2\dim\pnak{Q}{\dimvec, \dimvec - \uk; \framvec} = \dim\pnak{Q}{\dimvec; \framvec} + \dim\pnak{Q}{\dimvec - \uk; \framvec}.$$
The result now follows from Lemma \ref{lem: less than half}.
\end{proof}

\begin{remark}
    Note that if $\pnak{Q}{\dimvec, \dimvec - \uk; \framvec}$ is Lagrangian in $\pnak{Q}{\dimvec; \framvec} \times \pnak{Q}{\dimvec - \uk; \framvec}$ and $k_v \neq 0$; then Proposition \ref{prop: isotropic} implies that $Q$ does not have loops at $v$. 
\end{remark}

\subsection{Torus action}\label{sec:torus-action}

The torus action \eqref{eq:torus-action} on the usual quiver variety extends naturally to split parabolic quiver varieties. Recall the torus $\widetilde{T}_{\framvec} := \C^{\times} \times (\C^{\times})^{|Q_1|} \times \prod_{v \in Q_0}\widetilde{T}_v$. Then, $\widetilde{T}$ acts on $\Rep_{\double{\framed{Q}}}^{\uk}\left(\dimvec; \framvec\right)^{\st}_{0}$ as follows:
\[
(q, (t_e)_{e \in Q_1}, (U_v)_{v \in Q_0})\cdot (\widehat{A}, \widehat{B}, i, \widehat{j}) = (t_e\widehat{A}_e, q^{-1}t_e^{-1}\widehat{B}_e, q^{-1}iU_v^{-1}, U_v\widehat{j}).
\]
This action commutes with the action of $P_{\uk}$, so descends to an action on $\parnak{\uk}{Q}{\dimvec; \framvec}$ that factors through $T_{\framvec} = (\C^{\times}) \times (\C^{\times})^{|Q_1|} \times \prod_{v \in Q_0} T_v$. 

In terms of the variety $\pnak{Q}{\dimv{0}, \dots, \dimv{\ell}; \framvec}$, the torus $\widetilde{T}_{\framvec}$ acts on $\prequot{Q}{\dimv{0}, \dots, \dimv{\ell}; \framvec}$ as follows:
\[
(q, (t_e)_{e \in Q_1}, (U_v)_{v \in Q_0}) \cdot (A^{(\ast)}, B^{(\ast)}, i^{(\ast)}, j^{(\ast)}, \pi^{(\ast)}) = (t_eA_e^{(\ast)}, q^{-1}t_e^{-1}B_e^{(\ast)}, q^{-1}i_v^{(\ast)}U_v^{-1}, U_vj_v^{(\ast)}, \pi^{(\ast)}),
\]
and the action descends to an action of $T_{\framvec}$ on $\pnak{Q}{\dimv{0}, \dots, \dimv{\ell}; \framvec}$. Note that the map \eqref{eq:embedding} is $T_{\framvec}$-equivariant, where the $T_{\framvec}$-action on the right-hand side is diagonal. In particular, we have an embedding
\[
\pnak{Q}{\dimv{0}, \dots, \dimv{\ell}; \framvec}^{T_{\framvec}} \hookrightarrow \nak_Q(\dimv{0}; \framvec)^{T_{\framvec}} \times \cdots \times \nak_Q(\dimv{0}; \framvec)^{T_{\framvec}},
\]
so that $\pnak{Q}{\dimv{0}, \dots, \dimv{\ell}; \framvec}$ has isolated fixed points provided each of the varieties $\nak_Q(\dimv{0}; f), \dots,$ $\nak_Q(\dimv{\ell}; f)$ has isolated fixed points as well.

\subsection{Tautological bundles} By definition, the variety $\pnak{Q}{\dimv{0}, \dots, \dimv{\ell}; \framvec}$ comes equipped with tautological vector bundles $\CD_v^{(i)}$ of rank $\dimv{i}_v$. Moreover, these come equipped with surjective maps $\pi^{(i)}_{v}: \CD_{v}^{(i)} \to \CD_{v}^{(i+1)}$, and the kernel is a vector bundle of rank $\dimv{i}_{v} - \dimv{i+1}_{v}$.

Another way of obtaining these bundles is as follows. For each $i = 1, \dots, \ell$, we have a surjective map $\pi^{(i-1)}_v\circ \cdots \circ \pi_{v}^{(0)}: \CD_v^{(0)} \to \CD_v^{(i)}$, and the kernel $\CK_{v,(i)}$ is a vector bundle. For each $v \in Q_0$; these form a chain
\begin{equation}\label{eq: bundle-chain}
0 = \CK_{v,(0)} \subseteq \CK_{v,(1)} \subseteq \cdots \subseteq \CK_{v, (\ell)} \subseteq \CD_{v}^{(0)}
\end{equation}
and $\CD_v^{(i)} \cong \CD_v^{(0)}/\CK_{v, (i)}$.

Following the isomorphism given in Proposition \ref{prop: isomorphism}, for the variety $\parnak{\uk}{Q}{\dimvec; \framvec} \cong \pnak{Q}{\dimv{0}, \dots, \dimv{\ell}; \framvec}$ we denote the tautological bundle $\CD^{(0)}_{v}$ simply by $\CD_v$, and the bundle $\CD^{(\ell)}_{v}$ simply by $\widehat{\CD}_v$. Note that there is a surjective map $\pi_v: \CD_v \to \widehat{\CD}_v$. In this case, after eliminating redundant terms in \eqref{eq: bundle-chain}, we obtain a chain of sub-bundles
\[
0 = \CK_{v, 0} \subseteq \CK_{v, 1} \subseteq \cdots \subseteq \CK_{v, k_v} = \ker(\pi_v) \subseteq \CD_v
\]
 where $\mathrm{rank}(\CK_{v, i}) = i$. We have tautological line bundles
 \[
 \CL_{v, i} := \CK_{v, i}/\CK_{v, i-1}, i = 1, \dots, k_v.
 \]
Note that all of these bundles are equivariant with respect to the action of $T_{\framvec}$ on $\parnak{\uk}{Q}{\dimvec; \framvec}$.

\section{Quiver Brill-Noether loci}
\label{sec: BN}

\subsection{Definition and basic properties} 

For a vertex $v$ of $Q$, we denote by $\C_v$ the trivial one-dimensional representation of $\double{\framed{Q}}$ with dimension vector $\can_v$ and framing vector $0$. By analogy with \cite{Bayer}, we can consider the following subvariety in $\nak_Q(\dimvec;\framvec)$.

\begin{definition}
Let $\uk \in \Z_{\geq 0}^{Q_0}$. We define the \emph{quiver Brill-Noether locus}  as
$$
\BN_Q^{\uk}(\dimvec;\framvec):=
\left\{D\in \nak_Q(\dimvec;\framvec)\mid \dim\Hom(\C_v,D)\ge k_v\ \mathrm{for\ all}\ v\right\}. 
$$
Furthermore, we define
$$
\BN_Q^{=\uk}(\dimvec;\framvec):=
\left\{D\in \nak_Q(\dimvec;\framvec)\mid \dim\Hom(\C_v,D)=k_v\ \mathrm{for\ all}\ v\right\}. 
$$
\end{definition}

Clearly,
$$
\BN_Q^{\uk}(\dimvec;\framvec)=\bigsqcup_{\ur\geq \uk}\BN_Q^{=\ur}(\dimvec;\framvec).
$$

\begin{remark}
Several variants and special cases of the varieties $\BN^{\uk}_{Q}(\dimvec; \framvec)$ and $\BN^{=\uk}_{Q}(\dimvec; \framvec)$ have been considered before:
\begin{enumerate}
\item When the vector $\uk$ is concentrated at a single vertex, the variety $\BN^{=\uk}$ was considered by Nakajima in \cite[(4.3)]{nakajima-km}. We remark that Nakajima works with the dual stability condition from ours, so his varieties are defined in terms of the dimension of $\Hom(D, \C_v)$, see Remark \ref{rmk: stability}. The multiplicative analogue of these varieties was considered by Yamakawa in \cite{yamakawa}.
\item In the special case when $Q$ is the Jordan quiver and $\framvec = 1$, the varieties $\BN^{k}, \BN^{=k}$ were considered in \cite{Bayer,NakajimaYoshioka1,NakajimaYoshioka2}. We will elaborate on this in Example \ref{ex: hilbert-scheme} below. 
\end{enumerate}

\end{remark}

\begin{lemma}\label{lem:easy hom}
a) We have 
$$
\Hom(\C_v,D)=\bigcap_{\substack{e \in Q_1 \\ \source(e) = v}} \ker\left(A_e\right) \cap \bigcap_{\substack{e \in Q_1 \\ \target(e) = v}} \ker\left(B_e\right) \cap \ker\left(j_v\right).
$$

b) The locus $\BN_Q^{\uk}\left(\dimvec;\framvec\right)$ is closed in $\nak_Q\left(\dimvec;\framvec\right)$. 

c) The locus $\BN_Q^{=\uk}(\dimvec;\framvec)$ is open (but possibly empty) in $\BN_Q^{\uk}(\dimvec;\framvec)$.
\end{lemma}

\begin{proof}
Part (a) is clear, and part (b) follows from the semicontinuity of the dimension of the kernel. Part (c) follows from (b).
\end{proof}

Define the vector
\begin{equation}
\label{eq: def k0}
\uk^0[\dimvec]=\left(k_v^0[\dimvec]\right),\quad k^0_v[\dimvec]=2d_v-\sum_w (|Q_1(v,w)|+|Q_1(w,v)|) d_w-f_v=\sum_{w}\Cartan_{v,w}d_{w}-f_v.
\end{equation}

One can rewrite \eqref{eq: dim BN correspondence} as 
\begin{equation}
\label{eq: dim BN correspondence new}
\dim \pnak{Q}{\dimvec,\dimvec-\uk; \framvec}=\dim \nak_{Q}(\dimvec; \framvec)-\sum_{v}k_v\left(k_v-k^0_v[\dimvec]\right).
\end{equation}

\begin{lemma}
\label{lem: k bound}
Assume that $D\in \nak_Q(\dimvec;\framvec)$ and $\dim \Hom(\C_v,D)=k_v$. Then $k_v\ge k^0_v[\dimvec]$ for all $v$.
\end{lemma}
\begin{proof}
Consider the vector space 
$$
\Xi_v=\bigoplus_w D_w^{|Q_1(v,w)| + |Q_1(w,v)|} \oplus \bigoplus F_v.
$$
Note that 
$$
\dim \Xi_v=\sum_w (|Q_1(v,w)|+|Q_1(w,v)|) d_w+f_v.
$$
We have linear maps $M_1:D_v\to \Xi_v$ with components $(A_e,B_e,j_v)$ and $M_2:\Xi_v\to D_v$ with components $(B_e,-A_e,i_v)$. Note that $\Hom(\C_v, D) \cong \ker(M_1)$, cf. Lemma \ref{lem:easy hom}.  
The stability condition implies that $M_2$ is surjective, and the moment map equation implies $M_2M_1=0$.

We have $\dim \Imm M_2=d_v$, so $\dim\ker M_2=\dim \Xi_v-d_v$. Furthermore, $\Imm M_1\subset \ker M_2$, so $\dim \Imm M_1\le \dim \Xi_v-d_v$ and 
$$
k_v=\dim \ker M_1\ge d_v-(\dim \Xi_v-d_v)=2d_v-\dim \Xi_v=k^0_v[\dimvec].
$$
\end{proof}

\begin{remark}
    Note that the vector $\uk^0[\dimvec]$ may have negative coordinates. For example, if $Q$ is the Jordan quiver with framing $f$ then $\uk^{0}[d] = -f$. If $Q$ is the quiver with a single vertex and no arrows, $\uk^{0}[d] = 2d-f$, which may be negative. Note also that if $\affnak{Q}{\dimvec; \framvec}^{\reg}$ is nonempty (that is, if there exists a closed $G_{\dimvec}$-orbit with trivial stabilizer in $\mu^{-1}(0)$) then \cite[Lemma 4.7]{nakajima-km} implies that $\uk^{0}[\dimvec] \in \Z_{\leq 0}^{Q_0}$.
\end{remark}

\begin{lemma}
\label{lem: expected dimension}
Assume $k_v\ge k^0_v[\dimvec]$ for all $v$, then all components of $\BN_Q^{\uk}(\dimvec;\framvec)$ have dimension greater than or equal to \eqref{eq: dim BN correspondence}. If this expected dimension is achieved at each point, then $\BN_Q^{\uk}(\dimvec;\framvec)$ is Cohen–Macaulay.
\end{lemma}

\begin{proof}
We follow the proof of Lemma \ref{lem: k bound} and consider the matrix $M_1$ from $D_v$ to $\ker M_2$. We have $\dim \ker M_1\ge k_v$ if and only if $\rk M_1\le d_v-k_v$. The expected codimension of this locus is 
$$
(d_v-(d_v-k_v))(\dim \ker M_2-(d_v-k_v))=k_v(\dim \Xi_v-2d_v+k_v)=k_v\left(k_v-k^0_v[\dimvec]\right)
$$
and the expected codimension of the intersection of such loci over all $v$ is
$
\sum_{v} k_v(k_v-k^0_v[\dimvec]).
$
Now the result follows from \cite[Theorem 18.18]{Eis}.
\end{proof}

\subsection{The first projection}

Now we describe the relationship between the quiver Brill-Noether variety $\BN^{\uk}(\dimvec; \framvec)$ and the split parabolic quiver varieties, which is essentially given by the maps $p_{-}$ and $p_{+}$ from \eqref{eq: def p and q}.  First, we describe the geometric properties of the map $p_{-}$. Recall that $p_{-}$ sends a pair of representations $\left(D,\widehat{D}\right)$ to $D$.

\begin{lemma}
\label{lem: BN projection 1}
a) The image of $p_{-}$ is $\BN_Q^{\uk}(\dimvec;\framvec)$.
In particular, $\pnak{Q}{\dimvec,\dimvec-\uk; \framvec}$ is nonempty if and only if $\BN_Q^{\uk}(\dimvec;\framvec)$ is nonempty.

b) The fiber of $p_{-}$ over $\BN_Q^{=\ur}(\dimvec;\framvec)$ is isomorphic to $\prod_{v}\Gr(k_v,r_v)$. In particular, $p_{-}$ is proper.
\end{lemma}

\begin{proof}
Recall that $\pnak{Q}{\dimvec,\dimvec-\uk; \framvec}$ parametrizes short exact sequences of representations of $\double{\framed{Q}}$:
\begin{equation}
\label{eq: exact seq}
0\to K\hookrightarrow D\twoheadrightarrow \widehat{D}\to 0
\end{equation}
where $K$ is trivial and $\dim K_v=k_v$. Furthermore, by Lemma \ref{lem:stable-quotient} if $D$ is stable then $\widehat{D}$ is automatically stable, so it is enough to consider the embeddings $K\hookrightarrow D$.

Assume that $\dim \Hom(\C_v,D)=r_v$, then \eqref{eq: exact seq} is possible only if $r_v\ge k_v$. Conversely, if $r_v\ge k_v$ then a choice of $K$ is equivalent to a choice of subspaces
$$
K_v\subset 
\bigcap_{\substack{e \in Q_1 \\ \source(e) = v}} \ker\left(A_e\right) \cap \bigcap_{\substack{e \in Q_1 \\ \target(e) = v}} \ker\left(B_e\right) \cap \ker\left(j_v\right).
$$
Therefore the fiber of $p_{-}$ is isomorphic to
$$ 
\left(p_{-}\right)^{-1}(D)=\prod_v \Gr(k_v,r_v).
$$
\end{proof}

\begin{theorem}
\label{thm: BN projection 2}
Assume $\BN_Q^{=\uk}(\dimvec;\framvec)$ is not empty.
 Then:

a) The map $p_{-}$ is  an isomorphism over $\BN_Q^{=\uk}(\dimvec;\framvec)$.

b) The locus $\BN_Q^{=\uk}(\dimvec;\framvec)$ is smooth and dense in $\BN_Q^{\uk}(\dimvec;\framvec)$, and both have dimension \eqref{eq: dim BN correspondence new}.

c) The space $\BN_Q^{\uk}(\dimvec;\framvec)$ is  Cohen-Macaulay.
\end{theorem}

\begin{proof}
By Lemma \ref{lem: k bound} we have $k_v\ge k^0_v[\dimvec]$ for all $v$. By Lemma \ref{lem: BN projection 1} the map $p_{-}$ is an isomorphism over $\BN_Q^{=\uk}(\dimvec;\framvec)$ (which is open in $\BN_Q^{\uk}(\dimvec;\framvec)$), so $\BN_Q^{=\uk}(\dimvec;\framvec)$ is smooth of dimension
$$
\dim\BN_Q^{=\uk}(\dimvec;\framvec)= \dim\pnak{Q}{\dimvec,\dimvec-\uk; \framvec}.
$$
Furthermore, for $\ur\ge \uk$ we get 
$$
\dim \BN_Q^{=\ur}(\dimvec;\framvec)\le \dim\pnak{Q}{\dimvec,\dimvec-\uk; \framvec}-\sum_v k_v(r_v-k_v).
$$
In particular, all components of $\BN_Q^{\uk}(\dimvec;\framvec)$ have dimension at most $\dim \BN_Q^{=\underline{k}}(\dimvec;\framvec)=\dim\pnak{Q}{\dimvec,\dimvec-\uk; \framvec}$. By Lemma \ref{lem: expected dimension} this implies that all components of $\BN_Q^{\uk}(\dimvec;\framvec)$ have dimension exactly $\dim\pnak{Q}{\dimvec,\dimvec-\uk; \framvec}$, and thus $\BN_Q^{\uk}(\dimvec;\framvec)$ is Cohen-Macaulay.

Furthermore, for $r_v\ge k_v$ and $r_v\neq k_v$ the closure of $\BN_Q^{=\ur}(\dimvec;\framvec)$ has strictly smaller dimension, so its closure could not be an irreducible component. Therefore each irreducible component of $\BN_Q^{\uk}(\dimvec;\framvec)$ is the closure of an irreducible component of $\BN_Q^{=\uk}(\dimvec;\framvec)$. This implies that $\BN_Q^{=\uk}(\dimvec;\framvec)$ is dense in $\BN_Q^{\uk}(\dimvec;\framvec)$.
\end{proof}

\begin{remark}
    When $Q$ is the Jordan quiver Theorem \ref{thm: BN projection 2} is contained in \cite[Proposition 3.23]{NakajimaYoshioka2}. In this case, that the Brill-Noether locus is Cohen-Macaulay is also proved in \cite[Theorem 1.2]{Bayer}.
\end{remark}

\begin{definition}
Let $\underline{m},\underline{n}$ be two integer vectors in $\Z^{Q_0}$. We denote by $\max(\underline{m},\underline{n})$ the integer vector in $\Z^{Q_0}$ with components $\max(\underline{m},\underline{n})_v=\max(m_v,n_v)$.
\end{definition}

By Lemma \ref{lem: k bound} we have  
$$
\BN^{\uk}_Q(\dimvec;\framvec)=\BN^{\max\left(\uk,\uk^0[\dimvec]\right)}_Q(\dimvec;\framvec)
$$
and without loss of generality we can assume $\uk\ge \uk^0[\dimvec]$.

\begin{lemma}
\label{lem: dim preimage}
Assume $r_v\ge k_v\ge \max\left(k_v^0[\dimvec],0\right)$ for all $v$, and $\BN^{=\ur}_Q(\dimvec;\framvec)$ is not empty. Then
$$
\dim \left(p_{-}\right)^{-1}\left(\BN^{=\ur}_Q\left(\dimvec;\framvec\right)\right)<\dim \pnak{Q}{\dimvec,\dimvec-\uk; \framvec}
$$
unless $r_v=k_v$ for all $v$.
\end{lemma}

\begin{proof}
By Theorem  \ref{thm: BN projection 2} we have 
$$\dim \BN^{=\ur}_Q\left(\dimvec;\framvec\right)
=\dim \nak_Q(\dimvec;\framvec)-\sum_{v}r_v\left(r_v-k^0_v[\dimvec]\right).
$$
On the other hand, the fiber of $p_{-}$ over   $\BN^{=\ur}_Q(\dimvec;\framvec)$ is isomorphic to $\prod_v\Gr(k_v,r_v)$, so 
$$
\dim \left(p_{-}\right)^{-1}\left(\BN^{=\ur}_Q(\dimvec;\framvec)\right)=\dim \BN^{=\ur}_Q(\dimvec;\framvec)+\sum_v k_v(r_v-k_v)=
$$
$$
\dim \nak_Q(\dimvec;\framvec)-\sum_{v}r_v\left(r_v-k^0_v[\dimvec]\right)+\sum_v k_v(r_v-k_v)=
$$
$$
\dim \nak_Q(\dimvec;\framvec)-\sum_{v}k_v\left(k_v-k^0_v[\dimvec]\right)-\sum_{v}(r_v-k_v)\left(r_v-k_v^0[\dimvec]\right)=
$$
$$
\dim \pnak{Q}{\dimvec,\dimvec-\uk; \framvec}-\sum_{v}(r_v-k_v)\left(r_v-k_v^0[\dimvec]\right).
$$
By our assumption $r_v\ge k_v\ge k_v^0[\dimvec]$, so
\begin{equation}
\label{eq: r k equation}
\sum_{v}(r_v-k_v)\left(r_v-k_v^0[\dimvec]\right)\ge 0.
\end{equation}
The equality holds if and only if all terms in the left hand side of \eqref{eq: r k equation}  vanish. If $r_v=k_v^0[\dimvec]$ then $r_v=k_v$, so in fact $r_v=k_v$ for all $v$.
\end{proof}

\begin{theorem}
\label{thm: BN}
Assume $k_v\ge \max\left(k_v^0[\dimvec],0\right)$ for all $v$. Then:

a) $\BN^{=\uk}(\dimvec;\framvec)$ is nonempty if and only if $\BN^{\uk}(\dimvec;\framvec)$ is nonempty, in which case Theorem \ref{thm: BN projection 2} applies.

b) In particular, for $k_v= \max(k_v^0[\dimvec],0)$ the space $\BN^{=\uk}(\dimvec;\framvec)$ is open and dense in $\BN^{\uk}(\dimvec;\framvec)=\nak_Q(\dimvec;\framvec)$.

c) The space $\pnak{Q}{\dimvec,\dimvec-\uk; \framvec}$ is a resolution of singularities of $\BN_Q^{\uk}(\dimvec;\framvec)$.
\end{theorem}

\begin{proof}
We have
$$
\BN^{\uk}(\dimvec;\framvec)=\bigsqcup_{\ur\ge \uk}\BN^{=\ur}(\dimvec;\framvec)
$$
and 
\begin{equation}
\label{eq: preimage decomposition}
\pnak{Q}{\dimvec,\dimvec-\uk; \framvec}=\bigsqcup_{\ur\ge \uk}\left(p_{-}\right)^{-1}\left(\BN^{=\ur}(\dimvec;\framvec)\right).
\end{equation}
By Lemma \ref{lem: dim preimage} all the terms in \eqref{eq: preimage decomposition} with $\ur\neq \uk$ have positive codimension in $\pnak{Q}{\dimvec,\dimvec-\uk; \framvec}$. This implies that $\BN^{=\uk}(\dimvec;\framvec)$ is not empty and $\left(p_{-}\right)^{-1}\left(\BN^{=\uk}(\dimvec;\framvec)\right)$ is dense in  $\pnak{Q}{\dimvec,\dimvec-\uk; \framvec}$.

Let $W_1=\left(p_{-}\right)^{-1}\left(\BN^{=\uk}(\dimvec;\framvec)\right)$ and $W_2=\BN^{=\uk}(\dimvec;\framvec)$. By Theorem \ref{thm: BN projection 2} and the above, $W_1$ is open and dense in $\pnak{Q}{\dimvec,\dimvec-\uk; \framvec}$, $W_2$ is open and dense in $\BN^{\uk}(\dimvec;\framvec)$ and $p_{-}: W_1\to W_2$ is an isomorphism, so $p_{-}$ is birational. Since it is proper and all singular points of $\BN^{\uk}(\dimvec;\framvec)$ are contained in the complement of $W_2$, $p_{-}$ is a resolution of singularities.
\end{proof}

\begin{corollary}
Suppose $k_v\ge 0$ are arbitrary and $\pnak{Q}{\dimvec,\dimvec-\uk;\framvec}$ is not empty. Then the locus $\BN^{=\max\left(\uk,\uk^0[\dimvec]\right)}_Q(\dimvec;\framvec)$ is not empty and dense in the image of $p$. The fiber of $p$ over this dense subset is isomorphic to the product $\prod_{v}\Gr\left(k_v,\max\left(k_v,k^0_v[\dimvec]\right)\right)$.
\end{corollary}

\subsection{The second projection}

Next, we consider the second projection 
$$
p_{+}:\pnak{Q}{\dimvec,\dimvec-\uk;\framvec}\to \nak_Q(\dimvec-\uk;\framvec)
$$
which sends a pair $(D,\widehat{D})$ to $\widehat{D}$.

\begin{lemma}
\label{lem: second projection}
The image of $p_{+}$ is the quiver Brill-Noether locus 
$$\BN^{\max(\uk+\uk^0[\dimvec-\uk],0)}_{Q}(\dimvec-\uk;\framvec)
\subset \nak_Q(\dimvec-\uk;\framvec),$$ and the fiber of $p_{+}$ over $\BN^{=\ur}_Q(\dimvec-\uk;\framvec)$ is isomorphic to the product of Grassmannians 
$$
\prod_v \Gr\left(k_v,r_v-k^0_v[\dimvec-\uk]\right).
$$
In particular, the fiber of $p_{+}$  over
$\BN^{=\max(\uk+\uk^0[\dimvec-\uk],0)}_{Q}(\dimvec-\uk;\framvec)$ is the product of Grassmannians
$$
\prod_{v: k_v+k^0_v[\dimvec-\uk]<0} \Gr\left(k_v,-k^0_v[\dimvec-\uk]\right).
$$
\end{lemma}

\begin{proof}
Assume we have $\left(\widehat{A}, \widehat{B}, \widehat{i}, \widehat{j}\right) \in \nak_Q(\dimvec - \uk; \framvec)$.  We follow the notations of Lemma \ref{lem: k bound}, and for each vertex $v$ consider the maps $\widehat{M}_1, \widehat{M}_2$:
\begin{equation}
\label{eq: M1 M2 extension}
\begin{tikzcd}
D_v \arrow[dotted]{dd}{\pi_v} & &  \\
 & & \\
\widehat{D}_{v} \arrow{rr}{\widehat{M}_1} & & \widehat{\Xi}_v  \arrow[bend left]{ll}{\widehat{M}_2} \arrow[dotted,swap]{uull}{\widetilde{M}_2}
\end{tikzcd}
\end{equation}
where 
$$
\widehat{\Xi}_v =\bigoplus_w \widehat{D}_w^{|Q_1(v,w)| + |Q_1(w,v)|} \oplus \bigoplus F_v.
$$
We claim that the fiber of $p_{+}$ over $\left(\widehat{A}, \widehat{B}, i, \widehat{j}\right)$ is in correspondence with the collection of maps $$\left(\widetilde{M}_2: \widehat{\Xi}_v  \to D_v\right)_{v}$$ such that $\widetilde{M}_2$ is surjective and $\widetilde{M}_2\widehat{M}_1=0$. Indeed, assume that $(A, B, i, j) \in \left(p_{+}\right)^{-1}\left(\widehat{A}, \widehat{B}, \widehat{i}, \widehat{j}\right)$. By definition, the each map $D_w \to D_v$ factors through $\widehat{D}_w$, so this data defines a map $\widetilde{M}_2: \widehat{\Xi}_v \to D_v$. The stability condition on $(A, B, i, j)$ implies that $\widetilde{M}_2$ is surjective, and the moment map equation translates to $\widetilde{M}_2\widehat{M}_1 = 0$. On the other hand, assume a surjective map $\widetilde{M}_2:\widehat{\Xi}_v  \to D_v$ exists for every $v \in Q_0$, such that $\widetilde{M}_2\widehat{M}_1 = 0$. This is equivalent to finding a subspace $\mathscr{K}_v := \ker\left(\widetilde{M}_2\right)$ such that 
$$
\Imm \widehat{M}_1\subset \mathscr{K}_v\subset \ker \widehat{M}_2,
$$
so that $D_v = \widehat{\Xi}_v/\mathscr{K}_v$ and $\pi_v$ corresponds to the projection:
$$
\pi_v:D_v=\widehat{\Xi}_v/\mathscr{K}_v\twoheadrightarrow \widehat{\Xi}_v/\ker\widehat{M}_2=\widehat{D}_{v},
$$
in particular $\dim \ker \widehat{M}_2/\mathscr{K}_v=k_v$. With this, let us construct an element $(A, B, i, j) \in \nak_Q(\dimvec, \dimvec - \uk; \framvec)$ such that $p_{+}(A, B, i, j) = \left(\widehat{A}, \widehat{B}, \widehat{i}, \widehat{j}\right)$. The map $j_v$ is simply $\widehat{j}_v \circ \pi_v$. The map $i_v$ is the composition $F_v \hookrightarrow \widehat{\Xi}_v\twoheadrightarrow D_v$. If we have an arrow $e: w \to v$ in $Q$ then $A_{e}$ is the composition $D_w \twoheadrightarrow \widehat{D}_w \hookrightarrow \widehat{\Xi}_v \twoheadrightarrow D_v$, and  $B_{e^{\ast}}$ is the composition $D_v \twoheadrightarrow \widehat{D}_v \hookrightarrow \widehat{\Xi}_w  \twoheadrightarrow D_w$. That $(A, B, i, j)$ satisfies the moment map equations is equivalent to $M_2\widehat{M}_1 = 0$, and the representation $(A, B, i, j)$ is stable because $\widetilde{M}_2$ is surjective. Finally, by definition the kernel of $A, B$ and $j$ each contain the kernel of $D_v \twoheadrightarrow \widehat{D}_v$, so $(A, B, i, j)$ indeed defines an element of $\pnak{Q}{\dimvec, \dimvec - \uk; \framvec}$, and it is clear that $p_{+}(A, B, i, j) = \left(\widehat{A}, \widehat{B}, \widehat{i}, \widehat{j}\right)$. 

Now, let $\ur \in \Z_{\geq 0}^{Q_0}$ and assume that $\left(\widehat{A}, \widehat{B}, \widehat{i}, \widehat{j}\right) \in \BN_{Q}^{=\ur}(\dimvec - \uk; \framvec)$, so that  $\dim \ker \widehat{M}_1=r_v$. Then,
$$
\dim \ker \widehat{M}_2-\dim \Imm \widehat{M}_1=\left(\dim \widehat{\Xi}_v-(d_v-k_v)\right)-\left((d_v - k_v)-r_v\right)=r_v-k^0_v[\dimvec-\uk].
$$
Therefore the lift \eqref{eq: M1 M2 extension} is possible if and only if $r_v-k^0_v[\dimvec-\uk]\ge k_v$, and in this case the fiber is the product of Grassmannians
$$
\prod_v \Gr\left(k_v,r_v-k^0_v[\dimvec-\uk]\right).
$$
\end{proof}

\begin{lemma}
\label{lem: second projection dim preimage}
Assume that $\ur\ge \uk+\uk^0[\dimvec-\uk]$ and $\BN^{=\ur}_Q(\dimvec-\uk;\framvec)$ is nonempty. Then 
$$
\dim \left(p_{+}\right)^{-1}\left(\BN^{=\ur}_Q(\dimvec-\uk;\framvec)\right)<\dim \pnak{Q}{\dimvec,\dimvec-\uk;\framvec}
$$
unless $\ur=\max(\uk+\uk^0[\dimvec-\uk],0)$.
\end{lemma}

\begin{proof}
Observe that $\ur\ge \uk+\uk^0[\dimvec-\uk]$ implies 
$\ur\ge \uk^0[\dimvec-\uk]$, so Theorem \ref{thm: BN projection 2} applies to $\BN^{=\ur}_Q(\dimvec-\uk;\framvec)$. Now we need to compute the dimensions of all spaces in question.
Recall that 
$
\dim \nak_Q(\dimvec;\framvec)=2\framvec^T\dimvec-\dimvec^T\Cartan\dimvec$, so 
$$
\dim \nak_Q(\dimvec-\uk;\framvec)=2\framvec^T(\dimvec-\uk)-(\dimvec^T-\uk^T)\Cartan(\dimvec-\uk)=
$$
$$
\dim \nak_Q(\dimvec;\framvec)-2\framvec^T\uk-\uk^T\Cartan\uk+2\dimvec^T\Cartan\uk=\dim \nak_Q(\dimvec;\framvec)+2\uk^T\uk^0[\dimvec]-\uk^T\Cartan\uk.
$$
Note that 
$$
\uk^0[\dimvec-\uk]=\Cartan(\dimvec-\uk)-\framvec=\uk^0[\dimvec]-\Cartan\uk.
$$
Therefore
$$
\dim \BN^{=\ur}_Q(\dimvec-\uk;\framvec)=\dim \nak_Q(\dimvec-\uk;\framvec)-\sum_{v}r_v(r_v-\uk^0[\dimvec-\uk])=
$$
$$
\dim \nak_Q(\dimvec;\framvec)+2\uk^T\uk^0[\dimvec]-\uk^T\Cartan\uk-\ur^T(\ur-\uk^0[\dimvec]+\Cartan\uk),
$$
and finally
$$
\dim \left(p_{+}\right)^{-1}\left(\BN^{=\ur}_Q(\dimvec-\uk;\framvec)\right)=
\dim \nak_Q(\dimvec;\framvec)+2\uk^T\uk^0[\dimvec]-\uk^T\Cartan\uk-\ur^T(\ur-\uk^0[\dimvec]+\Cartan\uk)+\uk^T(\ur-\uk-\uk^0[\dimvec]+\Cartan\uk)=
$$
$$
\dim \nak_Q(\dimvec;\framvec)-\uk^T(\uk-\uk^0[\dimvec])-\ur^T(\ur-\uk-\uk^0[\dimvec]+\Cartan\uk)=
$$
$$
\dim \pnak{Q}{\dimvec,\dimvec-\uk;\framvec}-\sum_{v}r_v(r_v-k_v-k^0_v[\dimvec-\uk]).
$$
By our assumption $r_v-k_v-k^0_v[\dimvec-\uk]\ge 0$ and the result follows. 
\end{proof}

\begin{theorem}
\label{thm: irreducible}
Assume $\pnak{Q}{\dimvec,\dimvec-\uk;\framvec}$ is nonempty. Then both $\pnak{Q}{\dimvec,\dimvec-\uk;\framvec}$ and $\BN^{\uk}_Q(\dimvec;\framvec)$ are irreducible.
\end{theorem}

\begin{proof}
We proceed by induction on $|\dimvec|$. If $|\dimvec| = 0$ then necessarily $|\uk| = 0$ and $\pnak{Q}{\dimvec, \dimvec - \uk; \framvec} = \nak_Q(\dimvec; \framvec)$ is a point, so the base of induction follows. For the inductive step, assume we have $\dimvec \in \Z_{\geq 0}^{Q_0}$ and $\uk \leq \dimvec$ entry-wise. If $\uk=0$ then $\pnak{Q}{\dimvec,\dimvec-\uk;\framvec}$ is the Nakajima quiver variety $\nak_{Q}(\dimvec;\framvec)$ which is irreducible by a deep result of Crawley-Boevey \cite{crawley2001geometry}.

Assume now $\uk\neq 0$. By Lemma \ref{lem: second projection} we can write
$$
\pnak{Q}{\dimvec,\dimvec-\uk;\framvec}=\bigsqcup_{\ur\ge \uk+\uk^0[\dimvec-\uk]}\left(p_{+}\right)^{-1}\left(\BN^{=r}_Q(\dimvec-\uk;\framvec)\right).
$$
By Lemma \ref{lem: second projection dim preimage} the preimages of all strata with $\ur\neq \max(\uk+\uk^0[\dimvec-\uk],0)$ have positive codimension in $\pnak{Q}{\dimvec,\dimvec-\uk;\framvec}$. Therefore all irreducible components of $\pnak{Q}{\dimvec,\dimvec-\uk;\framvec}$ correspond to the irreducible components of $\BN^{=\max\left(\uk+\uk^0[\dimvec-\uk],0\right)}_Q(\dimvec-\uk;\framvec)$ or, equivalently (by Theorem \ref{thm: BN projection 2}) to  the irreducible components of $\BN^{\max\left(\uk+\uk^0[\dimvec-\uk],0\right)}_Q(\dimvec-\uk;\framvec)$. By the assumption of induction, the latter space is irreducible, so  $\pnak{Q}{\dimvec,\dimvec-\uk;\framvec}$  is irreducible. Now by Theorem \ref{thm: BN}  $\BN^{\uk}_Q(\dimvec;\framvec)$ is irreducible as well.
\end{proof}

\begin{remark}
This inductive argument is inspired by \cite[Theorem 3.3]{Bayer}.
\end{remark}

The proof of Theorem \ref{thm: irreducible} also implies the following combinatorial algorithm for determining whether $\pnak{Q}{\dimvec,\dimvec-\uk;\framvec}$ is nonempty.

\begin{corollary}
\label{cor: nonempty}
Suppose $\uk\neq 0$. The space $\pnak{Q}{\dimvec,\dimvec-\uk;\framvec}$ is nonempty if and only if $$\pnak{Q}{\dimvec-\uk,\dimvec-\uk-\max\left(\uk+\uk^0[\dimvec-\uk],0\right);\framvec}$$ is nonempty.
\end{corollary}

By applying Corollary \ref{cor: nonempty} inductively, we can decrease $\dimvec$ and eventually stop at $\uk=0$. In this case  $\pnak{Q}{\dimvec',\dimvec';\framvec}=\nak_Q(\dimvec';\framvec)$ and it is nonempty if and only if $\dimvec'$ is a root for $Q\cup\{\infty\}$ in the sense of \cite{crawley2001geometry}.

\subsection{More general projections}

In this section we describe the fibers of more general maps $P_+$ and $P_-$ from Lemma \eqref{lem: maps}. Recall the setup of Lemmas \ref{lem: k bound} and  \ref{lem: second projection} and consider the diagram
$$
\begin{tikzcd}
D_w \arrow[bend left]{rr}{M_{1,w}} \arrow[twoheadrightarrow,swap]{dd}{\pi_w} & & \Xi_w \arrow[twoheadrightarrow]{dd} \arrow {ll}{M_{2,w}}\\
 & & \\
\widehat{D}_w \arrow{rr}{\widehat{M}_{1,w}} \arrow[dotted,pos=0.3]{uurr}{\widetilde{M}_{1,w}}& & \widehat{\Xi}_w \arrow[bend left]{ll}{\widehat{M}_{2,w}} \arrow[dotted,pos=0.3,swap]{uull}{\widetilde{M}_{2,w}}
\end{tikzcd}
$$
We denote the dotted diagonal arrows by $\widetilde{M}_{1,w}$ and $\widetilde{M}_{2,w}$, so that 
$$
M_{1,w}=\widetilde{M}_{1,w}\pi_w,\ \widehat{M}_{2,w}=\pi_w\widetilde{M}_{2,w}.
$$

\begin{lemma}
\label{lem: fibers partial flags}
a) The fiber of $P_{-}$ is isomorphic to the product 
$$
\prod_{w}\Gr\left(d_w^{(\ell-1)}-d_{w}^{(\ell)},\ker M_{1,w}/K_w\right).
$$

b)  The fiber of $P_{+}$ is is isomorphic to the product 
$$
\prod_w\Gr\left(\dim \left[\ker \widetilde{M}_2/\Imm \widehat{M}_1\right]-(d^{(0)}_w-d^{(1)}_w),\ker \widetilde{M}_2/\Imm \widehat{M}_1\right).
$$
\end{lemma}

\begin{proof}

a) We have the map $\pi_w:D_w\twoheadrightarrow \widehat{D}_w$ with kernel $K_w=\ker \pi_w$ of dimension $d_w^{(0)}-d_w^{(\ell-1)}$ contained in $\ker M_{1,w}$. We would like to extend this to a projection
$\pi'_w:D_w\twoheadrightarrow \widehat{D}_w\twoheadrightarrow \widehat{D}'_w$ with kernel $K'_w=\ker \pi'_w$ such that $\dim K'_W=d_w^{(0)}-d_w^{(\ell)}$. As in Lemma \ref{lem: BN projection 1} we need $K_w\subset K'_w\subset \ker M_{1,w}$, so the choice of $K'_w$ is equivalent to the choice of a subspace $K_w'/K_w\subset \ker M_{1,w}/K_w$.

b) Here we instead want to extend $\pi_w$ to projection $\pi''_w: D''_w\twoheadrightarrow D_w\twoheadrightarrow \widehat{D}_w$. Similarly to the proof of Lemma \ref{lem: second projection}, this is equivalent to the choice of a subspace $\mathscr{K}_w$ such that 
$$
\Imm \widehat{M}_{1,w}\subset \mathscr{K}_w\subset \ker \widetilde{M}_{2,w}\subset \ker \widehat{M}_{2,w}
$$
and 
$$
D''_w=\widehat{\Xi}_w/\mathscr{K}_w\twoheadrightarrow D_w=\widehat{\Xi}_w/\ker \widetilde{M}_{2,w} \twoheadrightarrow \widehat{D}_w=\widehat{\Xi}_w/\ker \widehat{M}_{2,w}.
$$
Such $\mathscr{K}_w$ corresponds to a subspace $\mathscr{K}_w/\Imm \widehat{M}_1$ in $\ker \widetilde{M}_2/\Imm \widehat{M}_1$ of codimension
$$
\dim \ker \widetilde{M}_{2,w}-\dim \mathscr{K}_w=\dim D''_w-\dim D_w=d^{(0)}_w-d^{(1)}_w.
$$
\end{proof}

The next result will be important in the sequel \cite{BqtQ}. 
 
\begin{corollary}
\label{cor: fibers full flags}
Consider the maps \eqref{eq: def P plus minus}.

a) The fiber of $P_{-,v}$ is the projective space of lines in $\ker M_{1,v}/K_v$.

b)  The fiber of $P_{+,v}$ is the projective space of hyperplanes in $\ker \widetilde{M}_{2,v}/\Imm \widehat{M}_{1,v}$.
\end{corollary}

\subsection{Examples}
\begin{example}
Consider the $A_1$ quiver with dimension vector $d$ and framing $f$. The corresponding quiver variety $\nak_{Q}(d,f)$ is isomorphic to $T^*\Gr(f-d,f)$ with cotangent direction corresponding to the map $j: D\to \ker(i)$. Alternatively, we can denote $V_{f-d}=\ker(i)$ and $\alpha=ji\in \End(\C^f)$, then
$$
\alpha(\C^f)\subset V_{f-d},\ \alpha(V_{f-d})=0.
$$
The locus $\BN^k_Q(d;f)$ corresponds to the set of maps $j$ such that $\dim \ker j\ge k$ or equivalently $\rk j\le d-k$. For a fixed $D$ we get a determinantal variety of dimension $(d-k)(f-d+k)$, so overall 
$$
\dim \BN^k_Q(d;f)=d(f-d)+(d-k)(f-d+k)=(2d-k)f-2d(d-k)-k^2,
$$
in agreement with \eqref{eq: dim BN correspondence}. Note that in this case $k^0[d]=2d-f$ and $\rk j\le \min(d,f-d)$ so $k\ge d-\min(d,f-d)=\max(0,2d-f)=\max(0,k^0[d])$. For all such $k$ the locus $\BN^{=k}_Q(d;f)$ is nonempty.

A point in $\pnak{Q}{d,d-k;f}$ corresponds to the choice of a $k$-dimensional subspace in the kernel of $j$, which is indeed equivalent to a resolution of singularities of the determinantal variety. Alternatively, a point in $\pnak{Q}{d,d-k;f}$ corresponds to a flag
$
\left\{0\subset V_{f-d}\subset V_{f-d+k}\subset \C^f\right\}
$
and a map $\alpha\in \End(\C^f)$ such that
$$
\alpha(\C^f)\subset V_{f-d},\ \alpha(V_{f-d+k})=0.
$$
Finally, note that in this example
$$
k+k^0[d-k]=k+2(d-k)-f=2d-f-k\le 0,
$$
so the image of $p_+$ is $\BN^{0}_Q(d-k;f)=\nak_Q(d-k;f)$. The fiber of $p_+$ consists of subspaces $V_{f-d}$ such that $\Imm \alpha\subset V_{f-d}\subset V_{f-d+k}$, so indeed it is isomorphic to the Grassmannian of codimension $k$ hyperplanes in $V_{f-d+k}/\Imm \alpha$.
\end{example}

\begin{example}\label{ex: An single vertex}
    More generally, consider the equioriented $A_n$-quiver:
    
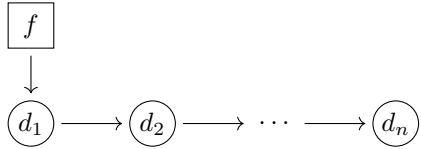
\begin{center}
\begin{tikzcd}
\node at (0.1,0) {f};
\draw (-0.2, -0.2) -- (-0.2, 0.4) -- (0.4, 0.4) -- (0.4, -0.2) -- cycle;

\draw[->] (0.1, -0.3) -- (0.1, -0.8);

\node at (0.1, -1.3) {d_1};
\draw (0.1, -1.2) circle (0.3);
\draw[->] (0.5, -1.2) -- (1.3, -1.2);
\node at (1.7, -1.3) {d_2};
\draw (1.7, -1.2) circle (0.3);
\draw [->] (2.1, -1.2) -- (2.9, -1.2);
\node at (3.3, -1.2) {\dots};
\draw[->] (3.7, -1.2) -- (4.5, -1.2);
\node at (4.9, -1.3) {d_n};
\draw (4.9, -1.2) circle (0.3);
\end{tikzcd}
\end{center}
The quiver variety $\nak_Q(\dimvec; \framvec)$ is nonempty if and only if $f \geq d_1 \geq \cdots \geq d_n$, in which case it is isomorphic to $T^{\ast}\Fl(f-d_1, \dots, f-d_n; f)$, as follows. An element $(A, B, i, j)$ is stable if and only if $i_1, A_1, \dots, A_{n-1}$ are surjective. We then produce a partial flag $0 \subseteq \ker(i_1) \subseteq \ker(A_1i_1) \subseteq \cdots \subseteq \ker(A_{n-1}\cdots A_1i_1) \subseteq \C^f$. To produce a cotangent vector to this flag, we consider $\alpha = j_1i_1 \in \End(\C^f)$. It follows easily from the moment map equations that $\alpha(\C^f) \subseteq \ker(A_{n-1}\cdots A_1i_1)$, and that $\alpha(\ker(A_m\cdots A_1i_1)) \subseteq \ker(A_{m-1}\cdots A_1i_1)$, while $\alpha(\ker(i_1)) = 0$, so that $\alpha$ is indeed a cotangent vector to the flag $0 \subseteq \ker(i_1) \subseteq \cdots \subseteq \ker(A_{n-1}\cdots A_1i_1) \subseteq \C^f$. 

We have
\[
k^0[\dimvec] = ((d_1 - d_2) - (f - d_1), (d_2 - d_3) - (d_1 - d_2), \dots, 2d_n - d_{n-1}).
\]
Note that $\dim(\ker(A_v)) = d_v - d_{v+1}$, so if $\BN^{=k}_{Q}(\dimvec; f) \neq \emptyset$, we must have $\max(0, k^0_v[\dimvec]) \leq k_v \leq d_v - d_{v+1}$. For any such $\uk$ the locus $\BN^{=\uk}$ is nonempty. The variety $\pnak{Q}{\dimvec, \dimvec - \uk; \framvec}$ can be identified with the following subbundle of $T^{\ast}\Fl(f - d_1, f-d_1+k_1, \dots, f-d_n, f-d_n+k_n;f)$:
\begin{equation}\label{eq:sub-bundle}
\Set{(0 \subseteq V_{f-d_1} \subseteq V_{f-d_1+k_1} \subseteq \cdots \subseteq V_{f-d_{n}} \subseteq V_{f-d_n+k_n} \subseteq \C^f, \alpha) | \begin{array}{l} \alpha(\C^f) \subseteq V_{f-d_n} \\\alpha(V_{f-d_v+k_v}) \subseteq V_{f-d_{v-1}} \\ \alpha(V_{f-d_1+k_1}) = 0\end{array}} 
\end{equation}
so $\pnak{Q}{\dimvec, \dimvec - \uk; \framvec}$ is indeed smooth. The map $p_{-}: \pnak{Q}{\dimvec, \dimvec - \uk; \framvec} \to \nak_{Q}(\dimvec; \framvec)$ forgets the subspaces $V_{f-d_i + k_i}$, while the map $p_{+}: \pnak{Q}{\dimvec, \dimvec - \uk; \framvec} \to \nak_{Q}(\dimvec - \uk; \framvec)$ forgets the subspaces $V_{f-d_i}$. 
\end{example}

\begin{example}
For the equioriented $A_n$-quiver with an arbitrary choice of framing vector, it follows from the main result of \cite{maffei} that $\pnak{Q}{\dimvec, \dimvec-\uk; \framvec}$ can be identified with the intersection of a parabolic Slodowy variety with a subbundle of the cotangent bundle to a partial flag variety defined similarly to \eqref{eq:sub-bundle}. More precisely, let $N := f_1 + 2f_2 + \cdots + nf_{n}$, and let $\mathbf{e} \in \mathfrak{sl}(N)$ be a nilpotent element of type $1^{f_1}2^{f_2}\cdots n^{f_{n}}$. Pick a semisimple element $\mathbf{h}$ and a nilpotent element $\mathbf{f}$ in $\mathfrak{sl}(N)$ such that $(\mathbf{e}, \mathbf{h}, \mathbf{f})$ form an $\mathfrak{sl}(2)$-triple. The Slodowy slice is the affine space $\mathcal{S}_{\mathbf{e}}:= \mathbf{e} + \ker(\mathrm{ad}(\mathbf{f})) \subseteq \mathfrak{sl}(N)
$. By convention, if $\mathbf{e} = 0$ then $\mathbf{h} = \mathbf{f} = 0$. 

Now, given the framing vector $\framvec$ and a dimension vector $\dimvec$ define the integers
\begin{equation}\label{eq:slodowy-data}
a_1(\dimvec) = f_1 + \cdots + f_{n} - d_1, \; a_i(\dimvec) = f_i + \dots + f_n + d_{i-1} - d_{i}  (i \leq 2 \leq n), \; a_{n+1}(\dimvec) = d_n.
\end{equation}
Note that $\sum_{i = 1}^{n+1} a_i(\dimvec) = N$. We denote the partial sums by $s_j = \sum_{i = 1}^{j} a_i$, $i = 1, \dots, n+1$. Then, by \cite[Theorem 1.2]{maffei}, $\nak_Q(\dimvec; \framvec)$ is nonempty if and only if $a_i(\dimvec) \geq 0$ for every $i = 1, \dots, n+1$, and in this case we have an isomorphism $\nak_Q(\dimvec; \framvec) \cong \widetilde{\mathcal{S}}_{\mathbf{e}}$, where $\widetilde{\mathcal{S}}_{\mathbf{e}}$ is the preimage of the Slodowy slice $\mathcal{S}_{\mathbf{e}}$ under the Springer map $\varphi: T^{\ast}\Fl(s_1, \dots, s_n; N) \to \mathfrak{sl}(N)$. 
Given the vector $\uk$, let $a'_1, \dots, a'_{n+1}$ be the sequence defined by \eqref{eq:slodowy-data} for the vector $\dimvec - \uk$, and let $s'_i$ be the corresponding partial sums. Assume that both $\nak_Q(\dimvec; \framvec)$ and $\nak_Q(\dimvec - \uk; \framvec)$ are nonempty. Then, $\nak_Q(\dimvec, \dimvec - \uk; \framvec)$ is nonempty if and only if we have
\[
s_1 \leq s'_1 \leq s_2 \leq s'_2 \leq \cdots \leq s_n \le s'_n.
\]
To describe $\nak_Q(\dimvec, \dimvec - \uk; \framvec)$ we consider the cotangent bundle $T^{\ast}\Fl(s_1, s'_1, s_2, s'_2, \dots, s_n, s'_n; N)$, and the subbundle
\begin{equation}\label{eq:Slodowy-subbundle}
\Set{(0 \subseteq V_{s_1} \subseteq V_{s'_1} \subseteq \cdots \subseteq V_{s_n} \subseteq V_{s'_n} \subseteq \C^N, \alpha) | \begin{array}{l} \alpha(\C^N) \subseteq V_{s_n} \\\alpha(V_{s'_i}) \subseteq V_{s_{i-1}} \\ \alpha(V_{s'_1}) = 0\end{array}},
\end{equation}
so that $\pnak{Q}{\dimvec, \dimvec - \uk; \framvec}$ is the intersection of \eqref{eq:Slodowy-subbundle} with the preimage of the Slodowy slice $\mathcal{S}_{\mathbf{e}}$ under the Springer map $T^{\ast}\Fl(s_1, s'_1, \dots, s_n, s'_n; N) \to \mathfrak{sl}(N)$. Note that when $\framvec$ is concentrated at the vertex $1$ we obtain $\mathbf{e} = 0$, so the Slodowy slice is $\mathcal{S}_{\mathbf{e}} = \mathfrak{sl}(N)$ and we recover Example \ref{ex: An single vertex}. It would be interesting to see if the embedding of \cite[Theorem 1.1]{dinkins} extends to the case of split parabolic quiver varieties.
\end{example}

\begin{example}\label{ex: hilbert-scheme}
When $Q$ is the Jordan quiver with $f=1$ the Brill-Noether loci were considered in \cite{Bayer,NakajimaYoshioka1,NakajimaYoshioka2}. 

Let $\C=\C[x,y]/(x,y)$ denote the one-dimensional module over $\C[x,y]$ supported at the origin.  Recall that $\nak_Q(d,1)$ is the Hilbert scheme of $d$ points on $\C^2$. Then
$$
\BN^{k}_Q(d,1)=\Set{I\in \Hilb^n(\C^2)|\dim \Hom_{\C[x,y]}\left(\C,\C[x,y]/I\right)\ge k}=
$$
$$
\Set{I\in \Hilb^n(\C^2)|\dim \C\otimes_{\C[x,y]}I\ge k+1}.
$$

The main result of \cite{Bayer} states that $\BN^{k}_Q(d,1)$ is irreducible and Cohen-Macaulay of dimension $2d-k(k+1)$, and non-empty whenever $2d-k(k+1)\ge 0$. Note that in this case the Cartan matrix $\Cartan$ vanishes, so $k^0[d]=-1$ and \eqref{eq: dim BN correspondence new} also yields
$$
\dim \BN^{k}_Q(d,1)=\dim \Hilb^d(\C^2)-k(k+1).
$$

The correspondence $\pnak{Q}{d,d-k; 1}$ was denoted by $\Hilb^{\dagger}_{d-k,d}(\C^2)$ in \cite[Definition 3.1]{Bayer} and by $\widehat{M}^0(c)$ in \cite{NakajimaYoshioka2}. It was proved in \cite[Remark 3.4]{Bayer} that it is irreducible, while \cite[Corollary 3.7]{NakajimaYoshioka2} proves that it is smooth. 
\end{example}

\begin{remark}\label{rmk: split-vs-philb}
    We note that for the Jordan quiver with framing $f = 1$, the split parabolic quiver variety $\parnak{k}{Q}{d; 1}$ differs from the parabolic flag Hilbert scheme $\mathrm{PFH}_{d, d-k}(\mathbb{C}^2)$ defined by the second author together with Carlsson and Mellit in \cite{CGM}.  Indeed, we have
    \[
    \parnak{k}{Q}{d; 1} = \{I_{d-k} \supseteq \cdots \supseteq I_{d} \supseteq (x,y)I_{d-k}\}, \qquad \text{but} \quad \mathrm{PFH}_{d, d-k} = \{I_{d-k} \supseteq \cdots \supseteq I_{d} \supseteq yI_{d-k}\}.
    \]
    Note that by \cite[Theorem 4.1.6]{CGM}, the space $\mathrm{PFH}_{d, d-k}$ is smooth of dimension $\dim\mathrm{PFH}_{d, d-k} = 2d - k > \dim \parnak{k}{Q}{d;1}$. Note that $\parnak{k}{Q}{d; 1} \subseteq \mathrm{PFH}_{d, d-k}$ and that, while the former space is isotropic in $\Hilb^{d}(\C^2) \times \Hilb^{d-k}(\C^2)$, the latter is Lagrangian. 
\end{remark}

\begin{example}\label{ex: gieseker}
More generally, when $Q$ is the Jordan quiver with framing $f$ the quiver variety $\nak_Q(d;f)$ is isomorphic to the Gieseker moduli space $\mathcal{M}(f,d)$ of framed rank $f$ sheaves on $\mathbb{P}^2$ with second Chern class $d$. 
Similarly, 
$$
\BN^{k}_Q(d,f)=\Set{\mathcal{F}\in \mathcal{M}(f,d)|\dim \Hom_{\C[x,y]}\left(\C,\C[x,y]^{\oplus f}/\mathcal{F}\right)\ge k}
$$
and by Theorem \ref{thm: BN projection 2} we get
$$
\dim \BN^{k}_Q(d,f)=2df-k(k+f).
$$
Here $k^0[d]=-f$. Similarly to Remark \ref{rmk: split-vs-philb}, the variety $\parnak{k}{Q}{d; f}$ differs from the parabolic Gieseker moduli space considered in \cite{GGS1}. 
\end{example}

\begin{example}
Example \ref{ex: hilbert-scheme} admits the following generalization. Let $\Gamma = \Z/r\Z$ be a cyclic group, with generator $\gamma$ acting on $\C^2$ by $\gamma.(x,y) = (\zeta x, \zeta^{-1} y)$, where $\zeta$ is a primitive $r$-root of unity. This action lifts to the Hilbert scheme $\Hilb^{\ast}(\C^2)$, and every component of the fixed point set $\Hilb^{\ast}(\C^2)^{\Gamma}$ is smooth.
Let $\widetilde{A}_r$ be the cyclic quiver with $r$ vertices labeled $(0, \dots, r-1)$, and consider the framing vector $\framvec = (1, 0, \dots, 0)$. Let $\dimvec$ be a dimension vector, and $n = d_0 + \dots + d_{r-1}$. Then, $\nak_{Q}(\dimvec; \framvec)$ can be identified with the component of $\Hilb^{n}(\C^2)^{\Gamma}$ consisting of $\Gamma$-invariant ideals $I \subseteq \C[x,y]$ such that, as $\Gamma$-modules,
\[
\C[x,y]/I \cong \bigoplus_{i = 0}^{r-1} \C_i^{d_i},
\]
where $\C_i$ is the $1$-dimensional representation of $\Gamma$ where $\gamma$ acts by $\zeta^{i}$, see e.g. \cite[Sect. 2.2]{negut-thesis}.

Let $\uk = (k_0, \dots, k_{r-1})$. To describe the Brill-Noether locus $\BN^{\uk}_{Q}(\dimvec; \framvec)$, consider the semidirect product algebra $\C[x,y]\# \Gamma$, and let $\C_i$ denote the $1$-dimensional representation of $\C[x,y]\#\Gamma$ where $x, y$ act by zero and $\gamma$ acts by $\zeta^{i}$. Then,
\[
\BN^{\uk}_{Q}(\dimvec; \framvec) = \Set{I \in \nak_{Q}(\dimvec; \framvec) | \dim \Hom_{\C[x,y]\#\Gamma}(\C_i, \C[x,y]/I) \geq k_i \; \forall i = 0, \dots, r-1}.
\]

This admits two generalizations. First, Example \ref{ex: gieseker} admits a similar generalization, cf. \cite[Section 2.2]{negut-thesis}. Second, one can replace $\Gamma$ with any finite subgroup of $\mathrm{SL}_2(\C)$,  in this case the quiver $Q$ is an orientation of the affine Dynkin diagram associated to $\Gamma$ via the McKay correspondence, see \cite{kuznetsov}, and the quiver Brill-Noether loci are given by bounding the dimensions of $\Hom_{\C[x,y]\#\Gamma}(S, \C[x,y]/I)$, where $S$ runs over the set of irreducible representations of $\Gamma$, which lift to an irreducible representation of $\C[x,y]\#\Gamma$ by letting $x, y$ act by zero. 
\end{example}

\section{Compositions of correspondences}\label{sec: compositions-correspondences}

In this section we discuss the relations between various spaces $\pnak{Q}{\dimvec^{(0)},\ldots,\dimvec^{(\ell)};\framvec}$ viewed as correspondences between quiver varieties, and compositions of such correspondences.

\begin{definition}
\label{def: excess}
Assume $A,B,C$ are smooth algebraic varieties and we have a commutative diagram:
\begin{equation}
\label{eq: excess siagram}
\begin{tikzcd}
 & \mathcal{M} \arrow{dr}\arrow{dl}& \\
 A\arrow{dr} & & B\arrow{dl}\\
 & C &
\end{tikzcd}
\end{equation}
We say that \eqref{eq: excess siagram} is a commutative square with excess $(R_1,R_2)$ if the following holds:
\begin{itemize}
\item[(a)] The fiber product $A\times_{C}B$ has excess bundle $\mathcal{F}$ of rank $R_1$ \cite[Chapter 6] {fulton_intersectionbook} and  
$\dim A\times_{C}B=\dim A+\dim B-\dim C+R_1$.
\item[(b)] There is a rank $R_2$ vector bundle $\mathcal{E}$ on the fiber product $A\times_{C}B$ with a section $s$ such that $\mathcal{M} \subset A \times_{C}B$ is the zero locus of $s$.
\item[(c)] $\dim \mathcal{M}=\dim A\times_{C}B-R_2$.
\end{itemize}

The goal of this section is to show that the diagram
\begin{equation}
\label{eq: composition}
\begin{tikzcd}
 & \pnak{Q}{\dimvec^{(0)},\dimvec^{(1)},\dimvec^{(2)};\framvec} \arrow{dr}{P_{+}}\arrow[swap]{dl}{P_-}& \\
 \pnak{Q}{\dimvec^{(0)},\dimvec^{(1)} ;\framvec}\arrow[swap]{dr}{P_+} & & \pnak{Q}{ \dimvec^{(1)},\dimvec^{(2)};\framvec}\arrow{dl}{P_-}\\
 & \nak_Q(\dimvec^{(1)};\framvec) &
\end{tikzcd}
\end{equation}
is a commutative square with excess $(R_1, R_2)$, where $R_1$ and $R_2$ are fixed integers that we define next. Define
\begin{equation}
 \uk=\dimvec^{(0)}-\dimvec^{(1)},\quad \uk'=\dimvec^{(1)}-\dimvec^{(2)},
 \end{equation}
\begin{equation}
\label{eq: def R1}
R_1=\sum_{v}\left(d^{(1)}_v-d^{(2)}_v\right)\left(d^{(0)}_v-d^{(1)}_v\right)=
\sum_{v}k_vk'_v
\end{equation}
and
\begin{equation}
\label{eq: def R2}
R_2=\sum_{v,w}\left(|Q_1(v,w)|+|Q_1(w,v)|\right) \left(d^{(1)}_v-d^{(2)}_v\right)\left(d^{(0)}_w-d^{(1)}_w\right)=\sum_{v,w}\left(|Q_1(v,w)|+|Q_1(w,v)|\right)k'_vk_w.
\end{equation}

\begin{theorem}\label{thm: excess}
The commutative square \eqref{eq: composition}
satisfies Definition \ref{def: excess} 
with excess $(R_1,R_2)$ defined by \eqref{eq: def R1} and \eqref{eq: def R2}.
\end{theorem}

In order to prove Theorem \ref{thm: excess} first we need to compute some dimensions.
\begin{lemma}
\label{lem: dim triple}
We have
$$
\dim \pnak{Q}{\dimvec^{(0)},\dimvec^{(1)},\dimvec^{(2)};\framvec}=\dim \pnak{Q}{\dimvec^{(0)},\dimvec^{(1)};\framvec}+\pnak{Q}{\dimvec^{(1)},\dimvec^{(2)};\framvec}-\dim \nak_Q\left(\dimvec^{(1)};\framvec\right)+R_1-R_2.
$$
\end{lemma}
\begin{proof}
By Lemma \ref{lem: less than half} we get
$$
\dim  \pnak{Q}{\dimvec^{(0)},\dimvec^{(1)} ;\framvec}=\frac{1}{2}\left(\dim \nak_{Q}(\dimvec^{(0)};\framvec)+\dim \nak_{Q}(\dimvec^{(1)};\framvec)\right)-\sum_{v,w}|Q_1(v,w)|k_vk_w,
$$
$$
\dim  \pnak{Q}{\dimvec^{(1)},\dimvec^{(2)} ;\framvec}=\frac{1}{2}\left(\dim \nak_{Q}(\dimvec^{(1)};\framvec)+\dim \nak_{Q}(\dimvec^{(2)};\framvec)\right)-\sum_{v,w}|Q_1(v,w)|k'_vk'_w,
$$
and
\begin{multline*}
\dim  \pnak{Q}{\dimvec^{(0)},\dimvec^{(1)},\dimvec^{(2)} ;\framvec}=\sum_{v}k_vk'_v+\frac{1}{2}\left(\dim \nak_{Q}(\dimvec^{(0)};\framvec)+\dim \nak_{Q}(\dimvec^{(2)};\framvec)\right)-\\
\sum_{v,w}|Q_1(v,w)|(k_v+k'_v)(k_w+k'_w).
\end{multline*}
The extra term $\sum_{v}k_vk'_v$ is the dimension of the fiber of the projection $$\pnak{Q}{\dimvec^{(0)},\dimvec^{(1)},\dimvec^{(2)} ;\framvec}\to \pnak{Q}{\dimvec^{(0)},\dimvec^{(2)} ;\framvec}.$$
This means that
\begin{multline*}
\dim  \pnak{Q}{\dimvec^{(0)},\dimvec^{(1)},\dimvec^{(2)} ;\framvec}=\dim  \pnak{Q}{\dimvec^{(0)},\dimvec^{(1)} ;\framvec}+\dim  \pnak{Q}{\dimvec^{(1)},\dimvec^{(2)} ;\framvec}-\dim \nak_Q(\dimvec^{(1)};\framvec)+\\
\sum_{v}k_vk'_v-\sum_{v,w}|Q_1(v,w)|(k_vk'_w+k'_vk_w).
\end{multline*}
Furthermore,
$$
\sum_{v,w}|Q_1(v,w)|(k_vk'_w+k'_vk_w)=\sum_{v,w}(|Q_1(v,w)|+|Q_1(w,v)|)k'_vk_w=R_2,
$$
and the result follows.
\end{proof}

Next, we consider the projection $P$ from the fiber product of $\pnak{Q}{\dimvec^{(0)},\dimvec^{(1)} ;\framvec}$ and $\pnak{Q}{ \dimvec^{(1)},\dimvec^{(2)};\framvec}$ to $\nak_Q\left(\dimvec^{(1)};\framvec\right)$.

\begin{lemma}
\label{lem: dim fiber product}
We have 
$$
\dim P^{-1}\left(\BN^{=\ur}_Q(\dimvec^{(1)};\framvec)\right)\le \dim \pnak{Q}{\dimvec^{(0)},\dimvec^{(1)},\dimvec^{(2)} ;\framvec}+R_2
$$
with equality for $\ur=\max\left(\uk',\uk+\uk^0\left[\dimvec^{(1)}\right]\right)$.
\end{lemma}

\begin{proof}
If the Brill-Noether locus is non-empty then we must have 
$$
r_v\ge k'_v, r_v\ge k_v+k^0_v[\dimvec^{(1)}].
$$
The fiber of $P$ is the product of fibers of $P_+$ and $P_-$, so by Lemmas \ref{lem: dim preimage} and \ref{lem: second projection dim preimage} we get
$$
\dim P^{-1}\left(\BN^{=\ur}_Q(\dimvec^{(1)};\framvec)\right)=\dim \BN^{=\ur}_Q(\dimvec^{(1)};\framvec)+\sum_{v}k'_v(r_v-k'_v)+\sum_v k_v\left(r_v-k_v-k^0_v\left[\dimvec^{(1)}\right]\right)=
$$
\begin{equation}
\label{eq:  fiber product step 1}
\dim \pnak{Q}{\dimvec^{(0)},\dimvec^{(1)};\framvec}-\sum_v r_v\left(r_v-k_v-k^0_v\left[\dimvec^{(1)}\right]\right)+\sum_{v}k'_v(r_v-k'_v).
\end{equation}
The last equation follows from the proof of Lemma \ref{lem: second projection dim preimage}.
Next, we need to relate the dimensions of $\pnak{Q}{\dimvec^{(0)},\dimvec^{(1)};\framvec}$ and $\pnak{Q}{\dimvec^{(0)},\dimvec^{(1)},\dimvec^{(2)};\framvec}$. Recall that
$$
\dim \pnak{Q}{\dimvec^{(0)},\dimvec^{(1)};\framvec}=\dim \nak_Q(\dimvec^{(0)};\framvec)-\sum_v k_v\left(k_v-k^0_v\left[\dimvec^{(0)}\right]\right),
$$
$$
\dim \pnak{Q}{\dimvec^{(0)},\dimvec^{(1)},\dimvec^{(2)};\framvec}=\dim \nak_Q(\dimvec^{(0)};\framvec)-\sum_v (k_v+k'_v)\left(k_v+k'_v-k^0_v\left[\dimvec^{(0)}\right]\right)+\sum_v k_vk'_v,
$$
therefore
$$
\dim \pnak{Q}{\dimvec^{(0)},\dimvec^{(1)},\dimvec^{(2)};\framvec}=\dim \pnak{Q}{\dimvec^{(0)},\dimvec^{(1)};\framvec}-\sum_v k'_v\left(k'_v-k^0_v\left[\dimvec^{(0)}\right]\right)-\sum_v k_vk'_v.
$$
This means that we can rewrite \eqref{eq:  fiber product step 1} as 
\begin{multline*}
\dim \pnak{Q}{\dimvec^{(0)},\dimvec^{(1)},\dimvec^{(2)};\framvec}+\sum_v k'_v\left(k'_v-k^0_v\left[\dimvec^{(0)}\right]\right)+\sum_v k_vk'_v+\\
\sum_{v}k'_v(r_v-k'_v)-\sum_v r_v\left(r_v-k_v-k^0_v\left[\dimvec^{(1)}\right]\right)=
\end{multline*}
\begin{equation}
\label{eq: fiber product step 2}
\dim \pnak{Q}{\dimvec^{(0)},\dimvec^{(1)},\dimvec^{(2)};\framvec}-\sum_v (r_v-k'_v)\left(r_v-k_v-k^0_v\left[\dimvec^{(1)}\right]\right)+\sum_v k'_v\left(2k_v-k^0_v\left[\dimvec^{(0)}\right]+k^0_v\left[\dimvec^{(1)}\right]\right).
\end{equation}
Now $\uk^0\left[\dimvec^{(1)}\right]=\uk^0\left[\dimvec^{(0)}\right]-\Cartan \uk$, so 
$$
\sum_v k'_v\left(2k_v-k^0_v\left[\dimvec^{(0)}\right]+k^0_v\left[\dimvec^{(1)}\right]\right)=\sum_{v,w}\left(|Q_1(v,w)|+|Q_1(w,v)|\right)k'_vk_w=R_2,
$$
so we get
$$
\dim \pnak{Q}{\dimvec^{(0)},\dimvec^{(1)},\dimvec^{(2)};\framvec}+R_2-\sum_v (r_v-k'_v)\left(r_v-k_v-k^0_v\left[\dimvec^{(1)}\right]\right)
$$
and the result follows.
\end{proof}

We are now ready to prove Theorem \ref{thm: excess}.

\begin{proof}[Proof of Theorem \ref{thm: excess}]
Let $\mathcal{Z}$ denote the fiber product $\pnak{Q}{\dimv{0}, \dimv{1}; \framvec} \times_{\pnak{Q}{\dimv{1}; \framvec}} \pnak{Q}{\dimv{1}, \dimv{2}; \framvec}$. It parametrizes flags of representations
$$
D^{(0)}\xrightarrow{\pi^{(0)}}D^{(1)}\xrightarrow{\pi^{(1)}} D^{(2)}
$$
such that $\pi^{(0)},\pi^{(1)}$ are surjective and commute with $A_e,B_e,i_v,j_v$ in appropriate sense. Furthermore, we require that $A_e,B_e,j_v$ vanish on $\ker\pi^{(0)}$ (as endomorphisms of $D^{(0)}$) and on $\ker \pi^{(1)}$ (as endomorphisms of $D^{(1)}$).   
This means that we have maps $A_e,B_e:D^{(1)}\to D^{(0)}$ which send $\ker \pi^{(1)}$ to $\ker \pi^{(0)}$. 

Note that $\pnak{Q}{\dimv{0}, \dimv{1}; \framvec}$ is cut out inside $\mathcal{Z}$ by the condition that $A_e, B_e, j_v$ vanish on $\ker(\pi^{(1)}\circ \pi^{(0)})$ (as endomorphisms of $D^{(0)}$). We note that $j_v$ automatically vanishes on $\ker(\pi^{(1)}\circ \pi^{(0)})$. Indeed, we have the following commutative diagram
\begin{center}
\begin{tikzcd}
D^{(0)}_v \arrow{rr}{j^{(0)}_v} \arrow[twoheadrightarrow, swap]{d}{\pi^{(0)}_v} & & F_v \\
D^{(1)}_v \arrow{urr}{j^{(1)}_v} \arrow[twoheadrightarrow, swap]{d}{\pi^{(1)}_v}& & \\
D^{(2)}_v \arrow{uurr}{j^{(2)}_v} & &
\end{tikzcd}
\end{center}
and it follows immediately that $j^{(0)}_v(\ker(\pi^{(1)}_v\circ \pi_v^{(0)})) = 0$. Thus, $\pnak{Q}{\dimv{0}, \dimv{1}; \framvec}$ is cut out inside $\mathcal{Z}$ only by the condition that $A_e, B_e$ vanish on $\ker(\pi^{(1)}\circ \pi^{(0)})$. 

As in Section \ref{sec: geometry}, we can describe $\mathcal{Z}$ explicitly by matrices $A_e,B_e$ which have the following block form:
\begin{equation}
\label{eq: block form}
\left(
\begin{tabular}{c|c|c} 
\quad \;\;\;\;  ? \qquad \qquad& \qquad 0 \qquad \qquad & \qquad 0 \qquad  \qquad\\
 \hline ?
 & \qquad 0 \qquad \qquad &  \qquad 0 \qquad \qquad\\
 \hline ?
  & \qquad $\star$ \qquad \qquad& \qquad 0 \qquad \qquad\\
\end{tabular}
\right)
\end{equation}
where columns are grouped in blocks of size $d^{(2)}_v,k'_v,k_v$ and rows are grouped in blocks of size $d^{(2)}_w,k'_w,k_w$ respectively. The $(k'_v\times k_w)$ block denoted by $\star$ defines a map $A_e^{\star}:\ker \pi^{(1)}\to \ker \pi^{(0)}$ (respectively $B_e^{\star}$).

The matrices $A_e,B_e, j_v, i_v$ should satisfy the moment map equation \eqref{eq:moment-map-vertex}
which also has the block form \eqref{eq: block form} but one can check, using the fact that $j_v$ already vanishes on $\ker(\pi^{(1)}\circ \pi^{(0)})$, that the block marked by $\star$ already vanishes. 

As a result, for each vertex $v$ we have $k_vk'_v$ fewer moment map equations than what one would expect from combining moment map equations for  $\pnak{Q}{\dimvec^{(0)},\dimvec^{(1)} ;\framvec}$ and $\pnak{Q}{ \dimvec^{(1)},\dimvec^{(2)};\framvec}$. In other words, the $\star$-marked block of the moment map equation forms an excess bundle $F$ of rank $\sum_v k_vk'_v=R_1$.

This implies that all components of $\mathcal{Z}$ have dimension at least 
$$
\dim \pnak{Q}{\dimvec^{(0)},\dimvec^{(1)};\framvec}+\pnak{Q}{\dimvec^{(1)},\dimvec^{(2)};\framvec}-\dim \nak_Q\left(\dimvec^{(1)};\framvec\right)+R_1.
$$
On the other hand, by Lemmas \ref{lem: dim triple} and \ref{lem: dim fiber product} there can be no components of larger dimension. 

Next, we note that $\pnak{Q}{\dimvec^{(0)},\dimvec^{(1)};\framvec}$ is cut out inside $\mathcal{Z}$ by the equations $A_e^{\star}=B_e^{\star}=0$, which is indeed equivalent to $A_e, B_e$ vanishing on $\ker(\pi^{(1)}\circ \pi^{(0)})$. We can write $\ker\pi^{(0)}=\bigoplus_v K_v, \ker\pi^{(1)}=\bigoplus_v K'_v$  and regard $K_v$ and $K'_v$ as vector bundles on the fiber product. Then the maps $A_e^{\star},B_e^{\star}$ define sections of the vector bundle $(K'_v)^{\vee}\otimes K_w$, and altogether these add up to a section of a vector bundle $\mathcal{E}$ of rank $R_2$ which cuts out  $\pnak{Q}{\dimvec^{(0)},\dimvec^{(1)},\dimvec^{(2)};\framvec}$ from $\mathcal{Z}$.  By applying Lemmas \ref{lem: dim fiber product} and \ref{lem: dim triple} again, we conclude that $\pnak{Q}{\dimvec^{(0)},\dimvec^{(1)},\dimvec^{(2)};\framvec}$ has the correct codimension $R_2$ in $\mathcal{Z}$.
\end{proof}

\begin{example}
Let $Q$ be the Jordan quiver, $d^{(0)}=d,d^{(1)}=d-1$ and $d^{(2)}=d-2$. Then, 
$$\dim \pnak{Q}{d,d-1;1}=2d-2 \;\; \text{and} \;\; \dim \pnak{Q}{d-1,d-2;1}=2d-4.$$
The fiber product over $\nak_Q(d-1;1)$ parametrizes triples of ideals $(I_d\subset I_{d-1}\subset I_{d-2}\subset \C[x,y])$ such that $I_{d-1}/I_d$ and $I_{d-2}/I_{d-1}$ are supported at the origin.  

The fiber product then has ``naive" expected dimension $(2d-2)+(2d-4)-(2d-2)=2d-4$.
On the other hand, on the fiber product the matrices $X$ and $Y$ are block triangular, with $(n-2)\times (n-2)$ square block and $2\times 2$ triangular block. Both of these  $2\times 2$ triangular blocks have zeros on diagonal, which gives 4 equations. However, $2\times 2$ triangular nilpotent matrices automatically commute, so we lose one equation from the moment map. This means that the true (virtual) dimension of the fiber product is $2d-3$.
 
On the other hand, $\dim \pnak{Q}{d,d-2;1}=2d-6$, and $\pnak{Q}{d,d-1,d-2;1}$ is a $\mathbb{P}^1$ bundle over $\pnak{Q}{d,d-2;1}$ hence $\dim  \pnak{Q}{d,d-1,d-2;1}=2d-5$. The variety $\pnak{Q}{d,d-1,d-2;1}$ is cut out in the fiber product by two equations corresponding to the off-diagonal terms in $2\times 2$ blocks.
\end{example}

\end{definition}

\bibliographystyle{plain}
\bibliography{bibliography.bib}

\end{document}